\documentclass[11pt, letterpaper]{amsart}
\usepackage{float}
\usepackage[utf8]{inputenc}

\usepackage{amssymb, amsmath, amsthm, wasysym, mathrsfs}
\usepackage{hyperref}
\usepackage[capitalize,nameinlink,noabbrev]{cleveref}
\usepackage[alphabetic,lite]{amsrefs}
\usepackage{tikz-cd}
\usepackage{comment}
\usepackage{thmtools}
\usepackage{setspace}
\usepackage{mathtools}
\usepackage{enumitem}
\usepackage{thm-restate}

\usepackage{caption}
\usepackage{subcaption}
\usepackage{float}
\usepackage{todonotes}

\newtheorem{lemma}{Lemma}[section]
\newtheorem{lemma*}{Lemma}
\newtheorem{theorem}[lemma]{Theorem}

\newtheorem{prop}[lemma]{Proposition}
\newtheorem{cor}[lemma]{Corollary}

\newtheorem{question}[lemma]{Question}
\newtheorem{claim*}{Claim}
\newtheorem{thm}[lemma]{Theorem}
\newtheorem{defn}[lemma]{Definition}
\newtheorem{example}[lemma]{Example}

\theoremstyle{definition}
\newtheorem{remark}[lemma]{Remark}

\newtheorem*{thm:thm:mainhakentheorem}{\cref{thm:mainhakentheorem}}

\theoremstyle{plain}

    \newtheoremstyle{TheoremNum}
        {\topsep}{\topsep}              
        {\itshape}                      
        {}                              
        {\bfseries}                     
        {.}                             
        { }                             
        {\thmname{#1}\thmnote{ \bfseries #3}}
    \theoremstyle{TheoremNum}

\newcommand{\Hh}{{\mathbb H}}

\newcommand{\C}{{\mathbb C}}
\newcommand{\F}{{\mathbb F}}

\newcommand{\Q}{{\mathbb Q}}
\newcommand{\R}{{\mathbb R}}
\newcommand{\Z}{{\mathbb Z}}

\DeclareMathOperator{\Aut}{Aut}

\DeclareMathOperator{\N}{N}

\DeclareMathOperator{\SL}{SL}
\DeclareMathOperator{\GL}{GL}

\DeclareMathOperator{\free}{free}

\DeclareMathOperator{\PSL}{PSL}

\numberwithin{equation}{section}
\numberwithin{table}{section}
\newcommand{\numCensus}{11031 }
\newcommand{\numHaken}{2423 }
\newcommand{\numHakenQHS}{2295 }
\newcommand{\numNoDihedralQuotient}{2009 }
\newcommand{\numDihedralQuotient}{286 }
\newcommand{\numWithIdeal}{1581 }
\newcommand{\numWithDimensionComputed}{608 }
\newcommand{\numWithDimensionOne}{29 }
\newcommand{\numWithDimensionZero}{579 }

\makeatletter
\renewenvironment{proof}[1][\proofname]{%
   \par\pushQED{\qed}\normalfont%
   \topsep6\p@\@plus6\p@\relax
   \trivlist\item[\hskip\labelsep\bfseries#1\@addpunct{.}]%
   \ignorespaces
}{%
   \popQED\endtrivlist\@endpefalse
}

\title{Detecting embedded surfaces using finite quotients}
\author{Tam Cheetham-West, Khánh Lê}
\date{Spring 2026}
\address{Department of Mathematics \\ Yale University \\ New Haven, CT, USA}
  \email{tamcheethamwest24@gmail.com}
\address{Department of Mathematical Sciences \\ Korea Advanced Institute of Science and Technology \\ Daejeon, South Korea}
\email{khanh.le.math@gmail.com}
\begin{document}
\pagestyle{plain}

\begin{abstract}
    We give conditions on a Haken hyperbolic rational homology three sphere that imply that any other 3-manifold with profinitely equivalent fundamental group must also be Haken. In the appendix, we show that a regular finite-sheeted cover of an aspherical integral homology three-sphere with positive first Betti number must have first Betti number at least four. We also show that this lower bound is sharp. 
\end{abstract}
\maketitle
\bibliographystyle{alpha}
\tableofcontents
\section{Introduction}
An oriented, irreducible, $\partial-$irreducible 3-manifold is \emph{Haken} if it contains an oriented essential embedded surface. There is a correspondence \cite{StallingsSurface} between essential embedded surfaces in an oriented, irreducible, $\partial-$irreducible 3-manifold $M$ and non-trivial splittings (graph of groups decompositions) of $\pi_1(M)$. Furthermore, non-trivial splittings of $\pi_1(M)$ correspond to non-trivial actions of $\pi_1(M)$ on simplicial trees (without inversions), see \cite[Section 5.3, Theorem 12]{SerreTrees}. 

\medbreak Two finitely-generated residually finite groups $\Gamma$ and $\Delta$ are said to be \emph{profinitely equivalent} if $\Gamma$ and $\Delta$ have the same finite quotients. Similarly, we say that two (compact) 3-manifolds $M$ and $N$ are profinitely equivalent if their fundamental groups $\pi_1(M)$ and $\pi_1(N)$ are profinitely equivalent. A Property $X$ of 3-manifolds is said to be a \emph{profinite property of 3-manifolds} if whenever the fundamental groups $\Gamma=\pi_1(M)$ and $\Delta=\pi_1(N)$ of two (compact) 3-manifolds $M,N$ are profinitely equivalent either $M$ and $N$ both have Property $X$ or $\Gamma$ and $\Delta$ both do not have Property $X$. Some examples of profinite properties of 3-manifolds are fibering \cite{JZ}, having a non-trivial JSJ decomposition \cite{WZ2}, and admitting a Thurston geometry \cite{WZ1}. This paper is concerned with the following question:
\begin{question}\label{hakenquestion}
Is Haken a profinite property of 3-manifolds?
\end{question}
A positive answer to the above question is implied by combining current results in the literature with Bridson-Reid's conjecture that finite-volume hyperbolic 3-manifolds are determined by the finite quotients of their fundamental groups (see \cref{structuralreduction} below, for example). However, it is known that for residually finite groups in general, admitting a non-trivial splitting is not a profinite property \cites{CLRS,bridsonfixedpt}. 

\begin{defn}
    A rational homology 3-sphere (abbreviated to $\Q$H$S^3$) is a compact 3-manifold $M$ for which $H_{*}(M,\Q)\cong H_{*}(S^3,\Q)$. By Poincaré duality, a compact 3-manifold $M$ is a rational homology sphere iff the first Betti number $b_1(M)=0$. 
\end{defn}

In this paper, we show that with additional hypotheses on $M$ a hyperbolic $\Q$H$S^3$, any hyperbolic $\Q$H$S^3$ profinitely equivalent to $M$ is Haken. 
More precisely, we prove:

\begin{thm}\label{thm:mainhakentheorem}
Let $M$ be a Haken, hyperbolic $\Q$H$S^3$ and let $N$ be a 3-manifold with $\widehat{\pi_1(M)}\cong\widehat{\pi_1(N)}$. If one of the following holds:
\begin{enumerate}
    \item $M$ contains an embedded virtual fiber, or
    \item there is a regular finite-sheeted cover $M'\to M$ with a non-trivial class $\alpha\in H_2(M';\Z)$ which is fixed up to sign by the deck group action, or 
    \item $\pi_1(M)$ admits an infinite dihedral quotient, or
    \item $M$ has a curve in its $(\mathrm{P})\SL_2(\F)$-character variety for $\F$ an algebraically closed field, or
    
    \item $\pi_1(M)$ has finitely many conjugacy classes of $\SL_2(\C)$ representations and at least one algebraic $\SL_2(\C)$ representation with non-integral trace.
    
\end{enumerate}
then $N$ is Haken. 
\end{thm}

Each condition in the statement of \cref{thm:mainhakentheorem} implies that $M$ is Haken. The proof of \cref{thm:mainhakentheorem} we offer involves showing that the collection of finite quotient of $\pi_1(M)$ detects all of the above conditions, and so it follows that $N$ must also be Haken.

\medbreak\noindent {\bf Organization of the paper.} In \cref{section:prelprofinite}, we record some facts about profinite completions of discrete groups that are used in the paper. We then discuss 3-manifolds in \cref{3mfldsprel}, highlighting existing results about the detection of essential surfaces and showing how \cref{hakenquestion} reduces to the case of hyperbolic $\Q $H$S^3$s. The next two sections form the core of the paper; in \cref{mainproof} we prove \cref{thm:mainhakentheorem} and we present examples of hyperbolic $\Q $H$S^3$s to which \cref{thm:mainhakentheorem} applies in \cref{section:examples}. We conclude in Section~\ref{conclusion} with an observation and question about the general case of \cref{hakenquestion}. We record more thoughts around Property $\mathcal{H}$ (introduced in \cref{mainproof}) in the Appendix. 

\medbreak\noindent {\bf Acknowledgments.} The authors thank Danny Ruberman, Ryan Spitler, and Matthew Stover for helpful conversations. The authors thank Xiaoyu Xu and Youheng Yao for their comments on the first draft of this paper. The authors are grateful for the mentorship and support of Alan Reid, whose questions and encouragement inspired this project from its inception. Part of this work appeared in the first author's Ph.D. thesis \cite{Cthesis}. 

\section{Preliminaries on profinite completions}\label{section:prelprofinite}

\subsection{Generalities}
The book \cite{RibesZalesskiiBook} is a standard reference for everything in this section. The profinite completion $\widehat{\Gamma}$ of a countable group $\Gamma$ is the inverse limit of the inverse system of its finite quotients $\Gamma/N$ for all finite-index normal subgroups $N\triangleleft\Gamma$ indexed by reverse-inclusion. The group $\widehat{\Gamma}$ is compact and totally-disconnected. It can be viewed as inheriting its topology from the product topology on $\prod_{N\in\mathcal{N}}\Gamma/N$ where $\mathcal{N}$ here denotes the set of all finite-index normal subgroups of $\Gamma$. There is a canonical group homomorphism $\Gamma\to\widehat{\Gamma}$ with dense image which is injective exactly when $\Gamma$ is residually finite. 

\medbreak The canonical map $\Gamma\to\widehat{\Gamma}$ gives a bijective correspondence between the finite index subgroups of $\Gamma$ and the open subgroups of $\widehat{\Gamma}$ that respects the lattice structure on finite-index subgroups and the conjugation action on finite-index subgroups, see \cite[Proposition 3.2.2]{RibesZalesskiiBook}. When $\Gamma$ is finitely generated, the group $\widehat{\Gamma}$ is {\it topologically} finitely generated (has a dense finitely generated subgroup). 
\begin{prop}[{\cite{DFPR},\cite[Corollary 3.2.8]{RibesZalesskiiBook}}]
    For finitely generated groups $\Gamma,\Delta$, $\widehat{\Gamma}\cong\widehat{\Delta}$ iff $\Gamma$ and $\Delta$ have the same set of finite quotients. 
    \begin{proof}
        
   We include the proof given in \cite{RibesZalesskiiBook} for completeness. Assume $\Gamma$ and $\Delta$ have the same set of (isomorphism classes of) finite quotients. Let $K_n$ be the intersection of all subgroups of index $\leq n$ in $\Gamma$, and let $L_n$ be the intersection of all subgroups of index $\leq n$ in $\Delta$. We can check that $\cap_nK_n=\{1\}$ and $\cap_n L_n=\{1\}$ and that the $\{K_n\}_{n\in\mathbb N},\{L_n\}_{n\in\mathbb N}$ form a nested sequence of characteristic subgroups that are {\it cofinal} in the sense of Lemma 1.1.9 of \cite{RibesZalesskiiBook} so that $\widehat{\Gamma}\cong\varprojlim_n\Gamma/K_n$ and $\widehat{\Delta}\cong\varprojlim_n\Delta/L_n$.
    \medbreak\noindent
    We first show that $\Gamma/K_n\cong\Delta/L_n$. To see this, we observe that $\Gamma/K_n$ is a finite quotient of $\Delta$ because $\Gamma$ and $\Delta$ have the same set of (isomorphism classes of) finite quotients. Then, we see that the kernel of the map $\Delta\twoheadrightarrow \Gamma/K_n$ will be a finite index normal subgroup of $\Delta$ containing $L_n$. In particular, there is an epimorphism $\Delta/L_n\twoheadrightarrow \Gamma/K_n$. If we repeat this process beginning with $\Delta/L_n$, we also obtain an epimorphism $\Gamma/K_n\twoheadrightarrow\Delta/L_n$ and so $\Gamma/K_n\cong\Delta/L_n$ for all $n\in\mathbb{N}$. 
    \medbreak\noindent Next, for every $n\in\N$, we form the set $X_n$ of isomorphisms between $\Gamma/K_n$ and $\Delta/L_n$. The collection of spaces $\{X_n\}_{n\in\N}$ has the structure of an inverse system of finite sets with maps $\Phi_{mn}:X_m\to X_n$ for all $m\geq n$ given as follows: for an isomorphism $\phi\in X_m$, $\phi(K_n)=L_n$, and so $\phi$ induces an isomorphism $\psi:\Gamma/K_n\to\Delta/L_n$, and we set $\Phi_{mn}(\phi)=\psi$. We then apply Proposition 1.1.4 of \cite{RibesZalesskiiBook} to see that $\varprojlim_nX_n$ is non-empty, and any element $\alpha\in\varprojlim_nX_n$ will give an isomorphism between $\varprojlim_n\Gamma/K_n$ and $\varprojlim_n\Delta/L_n$. Since $\widehat{\Gamma}\cong\varprojlim_n\Gamma/K_n$ and $\widehat{\Delta}\cong\varprojlim_n\Delta/L_n$, $\widehat{\Gamma}\cong\widehat{\Delta}$.
    \medbreak\noindent The other direction is given by the correspondence between finite index subgroups of a group and open subgroups of the profinite completion, see \cite[Proposition 3.2.2]{RibesZalesskiiBook}).
    \end{proof}
    
\end{prop}
Taking the profinite completion of a group is functorial and epimorphisms of finitely generated groups $G\twoheadrightarrow H$ can be completed to obtain continuous epimorphisms $\widehat{G}\twoheadrightarrow\widehat{H}$. A deep theorem of Nikolov-Segal \cite{ns} is that every finite-index subgroup of a finitely generated profinite group is open, equivalently, that every abstract homomorphism between finitely generated profinite groups is continuous.

\begin{lemma}
    For $\Gamma$ a finitely generated, residually finite group and $\Delta$ a finitely generated residually finite group with $\widehat{\Gamma}\cong\widehat{\Delta}$, $\Gamma_{ab}\cong\Delta_{ab}$.
    \begin{proof}
    Every finite abelian quotient of $\Gamma$ and $\Delta$ factors through their resepctive abelianizations. Thus, $\Gamma_{ab}$ and $\Delta_{ab}$ are finitely generated abelian groups with the same finite quotients. By the fundamental theorem of finitely generated abelian groups, $\Gamma_{ab}\cong\Z^n\oplus T$ and $\Delta_{ab}\cong\Z^m\oplus T'$ for $T,T'$ finite abelian groups. Choose a prime $p$ larger than the order of all torsion in $T,T'$. The largest finite quotient of $\Gamma_{ab}$ of the form $(\Z/p\Z)^k$ is $k=m$ and likewise the largest finite quotient of $\Delta_{ab}$ of the form $(\Z/p\Z)^k$ is $k=n$. Similarly, we can show that every summand of the primary decomposition of $T$ occurs in $T'$ and vice versa.
    \end{proof}
\end{lemma}
\begin{cor}
    \label{cor:profinitecompletiondetectsH1}
    For $\Gamma\cong\pi_1(M)$ and $\Delta\cong\pi_1(N)$, if $\widehat{\Gamma}\cong\widehat{\Delta}$ then $H_1(M;\Z)\cong H_1(N;\Z)$.
    \begin{proof}
    The first integral homology group is the abelianization of the fundamental group, see \cite[Theorem 2A.1]{hatcherAT}.  
    \end{proof}
\end{cor}

\subsection{The case of 3-manifold groups}

We conclude this section by summarizing some well-known results about the profinite completions of 3-manifold groups following \cite{reid2018profinite} as well as the recent result of Liu from \cite{LiuAlmostProfiniteRigidity} that we will use. We will also define some terminologies. 



\begin{defn}
    The profinite genus of a fixed group $\Gamma$ (resp. 3-manifold $M$) is the set of all groups (resp. 3-manifolds) profinitely equivalent to $\Gamma$ (resp. $M$). 
\end{defn}
Kneser and Milnor \cite{MilnorS2} prove that every closed 3-manifold $M$ has a unique (up to reordering) connect sum ({\it prime} or {\it Kneser-Milnor}) decomposition along embedded (reducing) 2-spheres into irreducible pieces (pieces in which every embedded 2-sphere bounds a 3-ball on one side) and copies of $(S^1\times S^2)$. Jaco-Shalen \cite{JS} and independently Johansson \cite{JSJ} further prove that every closed irreducible 3-manifold admits a canonical (JSJ) decomposition along essential embedded tori into pieces that are either Seifert-fibered or {\it atoroidal} i.e. not containing any essential embedded tori. Cornerstone results of Wilton-Zalesskii \cite{WZ2} establish that the profinite completions of the fundamental group of a closed 3-manifold $M$ determine the prime decomposition and the JSJ decomposition of $M$ when $M$ is irreducible. In particular,
 \begin{theorem}[{\cite[Theorem A ]{WZ2}}]
 Let $M$ and $N$ be closed, oriented 3-manifolds with $\widehat{\pi_1(M)}\cong\widehat{\pi_1(N)}$. If $M$ has a Kneser-Milnor decomposition $M\cong M_1\#  \dots \#M_n \#^r(S^1\times S^2)$ then $N$ has a Kneser-Milnor decomposition $N\cong N_1\#\dots \# N_n\#^r(S^1\times S^2)$ with $\widehat{\pi_1(M_i)}\cong\widehat{\pi_1(N_i)}$ for $i=1,\dots,n$. 
 \end{theorem}
 \begin{theorem}[{\cite[Theorem B]{WZ2}}]
   Let $M$ and $N$ be closed, oriented, irreducible 3-manifolds with $\widehat{\pi_1(M)}\cong\widehat{\pi_1(N)}$. The underlying graphs of the JSJ decompositions of $\pi_1(M)$ and $\pi_1(N)$ are isomorphic and corresponding vertex groups have isomorphic profinite completions. 
 \end{theorem}
 \medbreak\noindent 
The fundamental group $\Gamma=\pi_1(M)$ of a closed hyperbolic 3-manifold $M$ is a {\it uniform} lattice in $\PSL_2(\C)\cong$Isom$^+(\Hh^3)$, and by Mostow rigidity there is a unique discrete faithful representation of $\Gamma$ into $\PSL_2(\C)$ (up to conjugation). It is a useful consequence of the Virtual Special Theorem \cite{AgolHaken} that every hyperbolic $\Q$H$S^3$ has lots of finite-sheeted (regular) covers with positive $b_1$. To such covers, we will apply the important $\widehat{\Z}^\times-$regularity theorem of Liu \cite{LiuAlmostProfiniteRigidity}; a key ingredient used in the proof of the hyperbolic profinite almost rigidity theorem, \cite[Theorem 1.1]{LiuAlmostProfiniteRigidity}.

\begin{theorem}[{\cite[Theorem 1.2]{LiuAlmostProfiniteRigidity}}]
\label{thm:zhatregularity}
    Let $M$ and $N$ be finite-volume hyperbolic 3-manifolds. For any isomorphism $\Phi:\widehat{\pi_1(M)}\to\widehat{\pi_1(N)}$, let $\Phi_*:\widehat{H}_1(\pi_1(M),\Z)_{free}\to \widehat{H}_1(\pi_1(N),\Z)_{free}$ be the isomorphism induced on the profinite completion of the free part of $H_1(\pi_1(M),\Z)$ and $H_1(\pi_1(N),\Z)$ respectively. There is a unit $\mu\in\widehat{\Z}^\times$ and a module isomorphism $$s:H_1(\pi_1(M),\Z)_{free}\to H_1(\pi_1(N),\Z)_{free}$$ with completion $\widehat{s}:\widehat{H}_1(\pi_1(M),\Z)_{free}\to \widehat{H}_1(\pi_1(N),\Z)_{free}$ such that $\Phi_*=m_\mu\circ\widehat{s}$ where $m_\mu$ is left multiplication by $\mu$ on $\widehat{H}_1(\pi_1(N),\Z)_{free}$. 
\end{theorem}

An embedded essential surface $\Sigma$ in a closed Haken hyperbolic 3-manifold $M$ is one of two types. A {\it geometrically infinite} surface $\Sigma$ is a {\it virtual fiber}, a surface that will lift to the fiber $\Sigma'$ of a fibration $\Sigma'\hookrightarrow M'\to S^1$ in a finite-sheeted cover $M'\to M$, while a {\it geometrically finite} surface $\Sigma$ will be quasiFuchsian, quasiconformally conjugate to a lattice in $\PSL_2(\R)<\,\,\PSL(\C)$. One way to see the dichotomy is to consider the {\it limit set} (the set of accumulation points of an orbit of a point $x\in\Hh^3$ under the isometric action) in $\partial \Hh^3\cong S^2$ of a representative subgroup in $\PSL_2(\C)$ of a surface in the hyperbolic 3-manifold. A virtual fiber will have full limit set $S^2$ while the limit set of the quasiFuchsian surface will be a topological circle. 

\section{Reducing to the case of hyperbolic $\Q$H$S^3$s}\label{3mfldsprel}
 
Among compact, irreducible, $\partial$-irreducible 3-manifolds that are not hyperbolic $\Q$H$S^3$s, the fact that being Haken is a profinite property follows from prior results in the literature. As such, the general question of whether the Haken property is a profinite invariant reduces to the case of hyperbolic $\Q$H$S^3$s, as in the setting of \cref{thm:mainhakentheorem}. First, we consider the case where $M$ has positive $b_1$.
\begin{lemma}
    Let $M$ be a compact, irreducible 3-manifold with $b_1(M)>0$, if $N$ is a compact 3-manifold with $\widehat{\pi_1(M)}\cong\widehat{\pi_1(N)}$, then $N$ is Haken.
    \begin{proof}
        From $\widehat{\pi_1(M)}\cong\widehat{\pi_1(N)}$, it follows that $H_1(M,\Z)\cong H_1(N,\Z)$ by \cref{cor:profinitecompletiondetectsH1}, and therefore $b_1(N)>0$. The 3-manifold $N$ is irreducible by Theorem A \cite{WZ2}, and $N$ is Haken ( \cite[Lemma 6.6]{Hempel3mfds}).
    \end{proof}
\end{lemma}

\noindent The following theorem handles the case where the 3-manifold $M$ is a $\mathbb{Q}HS^3$ and is not hyperbolic. 

\begin{theorem}\label{structuralreduction}
    Let $M$ be a compact, irreducible, $\partial-$irreducible $\Q $H$S^3$ that is not hyperbolic. For $N$ a 3-manifold with $\widehat{\pi_1(M)}\cong\widehat{\pi_1(N)}$, $M$ is Haken iff $N$ is Haken.

\begin{proof}
Since $M$ is non-hyperbolic, $M$ is either geometric in the sense of Thurston or $M$ admits a non-trivial JSJ decomposition along essential tori. If $M$ admits a non-trivial JSJ decomposition along essential tori, by Theorem B \cite{WZ2} $N$ also admits a non-trivial JSJ decomposition. In particular, $N$ contains an essential torus and is Haken.

\medbreak\noindent We can therefore assume that $M$ is geometric and non-hyperbolic. Thus, the manifold $M$ is either Seifert-fibered or has Sol geometry. Since $M$ is a $\Q$H$S^3$, $M$ must be a closed 3-manifold. Otherwise by Poincare duality, we have $b_1(M) = \frac{1}{2}b_1(\partial M) > 0$. When $M$ is a closed Seifert-fibered 3-manifold, by Theorem 1.2 \cite{wilkesSFS} $M$ is either completely distinguished by the profinite completion of $\pi_1(M)$ or $M$ is one of the examples of Hempel \cite{Hempel}. In the first case, $N$ is homeomorphic to $M$, and in the second case $N$ is the mapping torus of a periodic surface homeomorphism; in either case, $N$ is Haken. 

 \medbreak\noindent When $M$ has Sol geometry, $N$ has Sol geometry as well, and $N$ is a Haken manifold. This is because $N$ has a finite sheeted regular cover $N'$ that is a torus bundle over $S^1$ with Anosov monodromy whose linear representative is a matrix $A$ in $\SL_2(\Z)$ with $|tr(A)|>2$, see \cite[Proposition 12.7.6]{Martelli} for example. By a straightforward computation, we see that $b_1(N') = 1$, see \cite[Lemma 3.5]{BR} for example. By Poincaré duality, we get $H_2(N';\Z) \cong \Z$. A generator for $H_2(N',\Z)$ is fixed up-to-sign by the finite deck group action and therefore we can apply Theorem 3.1 \cite{HassHomology} to conclude that $N$ is Haken.


\end{proof}
\end{theorem}

\section{Proof of the main theorem}\label{mainproof}
For the rest of this section let $M$ be a closed Haken, hyperbolic $\Q$H$S^3$ and $N$ be a 3-manifold such that $\widehat{\pi_1(N)}\cong\widehat{\pi_1(M)}$. We fix an isomorphism $\Phi:\widehat{\pi_1(N)}\to\widehat{\pi_1(M)}$. Following Liu, we say that $(M,N;\Phi)$ is a \emph{profinite pair in the Haken setting}.  

\subsection{The relationship between conditions in \cref{thm:mainhakentheorem} and Haken}

Before embarking on the proof of \cref{thm:mainhakentheorem}, we discuss the relationships among the conditions in \cref{thm:mainhakentheorem} as well as the Haken property of 3-manifolds. Most implications are well-known results from 3-manifold topology with the exception of the equivalence between (2) and (3) whose proof will be given in this section.    

If $M$ contains an embedded virtual fiber, then $M$ is a union of two twisted $I$-bundles over a non-orientable surface \cite{Purcell}. It follows from this decomposition of $M$ that $\pi_1(M)$ maps onto the infinite dihedral group $D_\infty$. Hempel and Jaco proved a partial converse of the previous implication assuming that the kernel of the surjective homomorphism onto $D_\infty$ has finitely generated kernel, see \cite{HempelJacoExtensions}, also \cite[Theorem 11.1]{Hempel3mfds}. 

Condition (2) in \cref{thm:mainhakentheorem} is taken from the work \cite{HassHomology} of Hass who showed that if a 3-manifold $M$ satisfies (2) then it is Haken. For convenience, we will use the following  terminology. 

\begin{defn}\label{def:PropertyH}
A closed 3-manifold $M$ has Property $\mathcal{H}$ if it has a regular, finite-sheeted cover $M'$ with a non-trivial fixed class (up to sign) in $H_2(M',\Z)$ under the induced action of the deck group $\pi_1(M)/\pi_1(M')$.
\end{defn}

\noindent We make the following observations about Property $\mathcal{H}$ in the case of $\mathbb{Q}HS^3$ which will be used in the subsequent discussions. The remark below was communicated to us by Danny Ruberman. 

\begin{remark}
    \label{rem:Ruberman}
    Suppose that for some $\beta \in H_2(M';\mathbb{Z})\setminus \{0\}$ we have $g_*\beta=\beta$ for all $g\in G$ in the setup of \cref{def:PropertyH}. It follows that $\beta$ is a non-trivial homology class in the image of the transfer homomorphism $H_2(M;\Z)\to H_2(M';\Z)$. In particular, $M$ would not be a $\Q HS^3$ and would be Haken.  
\end{remark}

\begin{prop}
\label{prop:EffectivePropertyH}
    Suppose that $M$ is a $\mathbb{Q}HS^3$. Then $M$ has Property $\mathcal{H}$ if and only if  $\pi_1(M)$ admits the infinite dihedral group $D_\infty$ as a quotient. 
\end{prop}

\begin{proof}
    Suppose that there exists a group epimorphism $\rho: \pi_1(M) \to D_\infty$ where
    \begin{equation}
    \label{eq:InfiniteDihedralPresentation}
    D_\infty  = \langle a,b \mid a^2, aba^{-1} = b^{-1}\rangle.
    \end{equation}
    Let $M'\to M$ be the degree-two cover such that $\pi_1(M') = \rho^{-1}(\langle b \rangle)$. The restriction of $\rho$ to $\pi_1(M')$ gives a cohomology class $\phi:\pi_1(M')\to\Z$ in $H^1(M';\Z)$ such that $\rho(\gamma) = b^{\phi(\gamma)}$ for all $\gamma \in \pi_1(M')$. Let $t \in \pi_1(M)$ such that $\rho(t) = a$. Since $\rho(t\gamma t^{-1}) = a\rho(\gamma) a^{-1} = \rho(\gamma)^{-1}$ for all $\gamma \in \pi_1(M')$, we have 
    \[
    \phi(t\gamma t^{-1}) = -\phi(\gamma). 
    \]
    In particular, the action of the deck group $\pi_1(M)/\pi_1(M') \cong \Z/2\Z$ negates $\phi$. The Poincare dual $\alpha \in H_2(M';\mathbb{Z})$ of $\phi \in H^1(M';\Z)$ gives an element that is fixed up to sign by $\pi_1(M)/\pi_1(M')$. Therefore, $M$ has Property $\mathcal{H}$.    
    
    Suppose that $M$ has Property $\mathcal{H}$. Then there exists a finite regular cover $M_G \to M$ with deck group $G = \pi_1(M)/\pi_1(M_G)$, and a non-trivial class $\alpha \in H_2(M;\Z)$ that is fixed by $G$ up to sign. In particular, $G$ acts on $\{\pm \alpha\}$ nontrivial. Otherwise, the triviality of this action would contradict the fact that $M$ is a $\mathbb{Q}HS^3$, see \cref{rem:Ruberman}. 
    
    Let $H$ be the index-two subgroup of $G$ that fixes $\alpha$ and $M_H$ be a degree-two cover of $M$ that corresponds to $H$. The cover $M_G \to M_H$ is regular, of finite index and has deck group $H = \pi_1(M_H)/\pi_1(M_G)$. Let $\beta \in H_2(M_H;\Z)$ be the image of $\alpha \in H_2(M_G;\Z)$ under the induced homomorphism of the covering map $M_G \to M_H$.
    Since the deck group $H$ fixes $\alpha$, the image of $\beta \in H_2(M_H;\Z)$ in $H_2(M_G;\Z)$ under the transfer homomorphism is nontrivial, see \cref{rem:Ruberman}. In particular, $\beta$ is a non-trivial in $H_2(M_H;\Z)$. Let $\beta^*:\pi_1(M_H) \to \Z$ be the Poincare dual to $\beta$ that is also primitive.

    We claim that there exists an epimorphism from $\pi_1(M)$ to the infinite dihedral group. We present $\pi_1(M) = \langle t, \pi_1(M_H) \rangle$ for some fixed $t \not\in\pi_1(M_H)$. The action of $\pi_1(M)/\pi_1(M_H)$ on $\beta$ is multiplication by $-1$ since it is also multiplication by $-1$ on $\alpha$. In particular, we have
    \[
    \beta^*(t\gamma t^{-1}) = -\beta^*(\gamma)
    \]
    for all $\gamma \in \pi_1(M_H)$. When $\gamma = t^2$, we get $\beta^*(t^2) = -\beta^*(t^2)$ which implies that $\beta^*(t^2) = 0$ since $\Z$ has no 2-torsion.
    
    Since $\pi_1(M_H)$ has index two and $t \not \in \pi_1(M_H)$, we can write elements of $\pi_1(M)$ uniquely as $t^{\varepsilon} \gamma$ for $\gamma \in \pi_1(M_H)$ and $\varepsilon \in \{0,1\}$. Using the presentation in \cref{eq:InfiniteDihedralPresentation}, we define an assignment
    \[
    t^\varepsilon \gamma\mapsto \rho(t^\varepsilon \gamma) = a^\varepsilon  b^{\beta^*(\gamma)} 
    \]
    for all $\gamma \in \pi_1(M_H)$. We check that this assignment defines a group homomorphism. Since the restriction of $\rho$ to $\pi_1(M_H)$ is just the group homomorphism $\beta^*$ written in multiplicative notation, it suffices to check that
    \[
    \rho(t\gamma t^{-1}) = b^{\beta^*(t\gamma t^{-1})} =  b^{-\beta^*(\gamma)} = ab^{\beta^*(\gamma)}a^{-1} = \rho(t)\rho(\gamma)\rho(t^{-1})
    \]
    for every $\gamma \in \pi_1(M_H)$ and that
    \[
    \rho(t^2) = b^{\beta^*(t^2)} = b^0 = 1.
    \]
    We get a group epimorphism $\rho:\pi_1(M) \to D_\infty$ since $\beta^*$ is chosen to be primitive. 
    If we further assume that $\alpha$ can be represented by a fiber class, then $\beta$ is a homomorphism with finitely generated kernel. 
\end{proof}

Now if $\pi_1(M)$ has an infinite dihedral quotient, then $\pi_1(M)$ has dihedral quotients $D_{2k}$ of any order $|D_{2k}| = 2k$ for $k \geq 1$. Since $D_{2k}$ can be realized as a subgroup of $\PSL_2(\mathbb C)$, it follows that the $\PSL_2(\mathbb{C})$-character variety of $\pi_1(M)$ contains a curve, see also the first proof of \cref{lem:EVF_Detection}. We note that having an infinite dihedral quotient does not imply that the $\SL_2(\mathbb{C})$-character varierty of $\pi_1(M)$ contains a curve since the finite dihedral representations need not lift to $\SL_2(\mathbb{C})$. 

Conditions (4) and (5) in \cref{thm:mainhakentheorem} are motivated by Culler--Shalen theory \cite{CullerShalen}. One of the main results of Culler--Shalen theory shows that if $\pi_1(M)$ has a curve in its $\SL_2(\mathbb{C})$-character variety then $M$ is Haken. Extending Culler--Shalen theory, Boyer and Zhang showed that similar condition for $\PSL_2(\mathbb{C})$-character variety also implies that $M$ is Haken \cite[Theorem 4.3]{BoyerZhang}. The fact that algebraic non-integral point on the character variety of $\pi_1(M)$ also corresponds to splittings of $\pi_1(M)$ can be found in \cite[Theorem 5.2.7]{MaclachlanReidBook} as a direct consequence of \cite{SerreTrees}. 

Finally, we remark that by \cref{prop:EffectivePropertyH}, Property $\mathcal{H}$ implies that there is a curve in the $\PSL_2(\mathbb{C})$-character variety of $\pi_1(M)$. Boyer--Zhang result then tells us that $M$ is Haken \cite[Theorem 4.3]{BoyerZhang}. In other words, this gives an alternative proof of \cite[Theorem 3.1]{HassHomology}.

We summarize the discussion above in the following diagram:

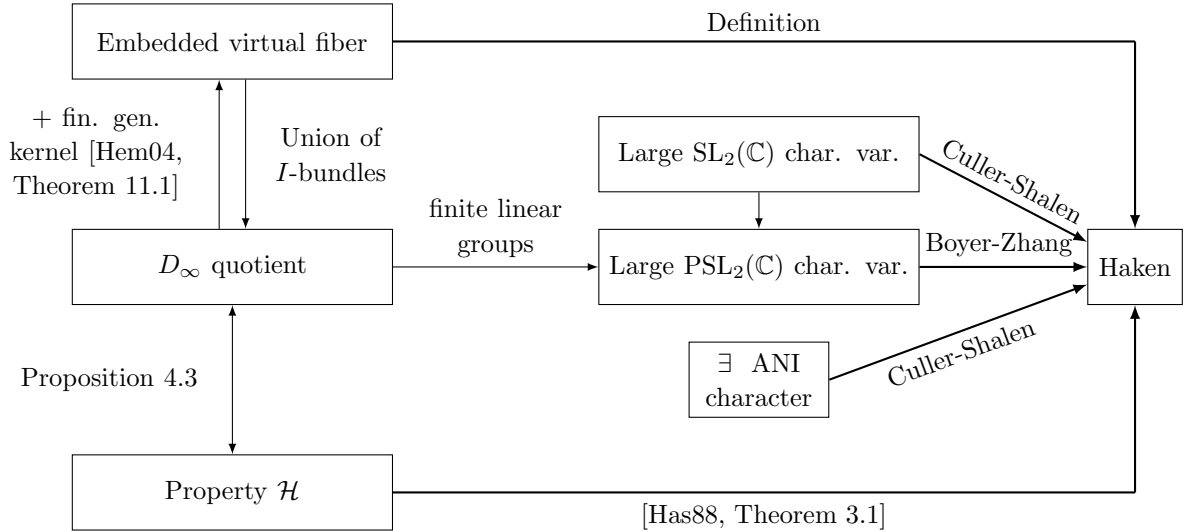
\begin{figure}[h!]
\begin{tikzpicture}[propnode/.style={draw, rectangle, minimum size=1cm}, node distance=5cm, align=center]

\node[propnode] (EVF) at (-6,2) [text width=4cm]  {\small Embedded virtual fiber};
\node[propnode] (Hass) at (-6,-4) [text width=4cm] {\small Property $\mathcal{H}$ };
\node[propnode] (Dihedral) at (-6,-1) [text width=4cm]{\small $D_\infty$ quotient};

\node[propnode] (SLCurve) at (1,0.5) [text width=4cm]{\small Large $\SL_2(\mathbb{C})$ char. var.};

\node[propnode] (PSLCurve) at (1,-1) [text width=4cm]{\small Large $\PSL_2(\mathbb{C})$ char. var.};

\node[propnode] (ANI) at (1,-2.5) [text width=1.6cm]{\small $\exists\ $ ANI \\ character};

\draw[-latex] ([xshift = 5]EVF.south) -- ([xshift = 5]Dihedral.north) node[midway, right, text width=2cm, align=center ] {\small Union of $I$-bundles}; 

\draw[-latex] ([xshift = -5]Dihedral.north) -- ([xshift = -5]EVF.south) node[midway, left, text width=3cm, align=center ] {\small + fin. gen. kernel \cite[Theorem 11.1]{Hempel3mfds}}; 

\draw[latex-latex] (Hass) -- (Dihedral) node[midway, left, text width=3cm, align = center] {\small \cref{prop:EffectivePropertyH}}; 

\draw[-latex] (Dihedral) -- (PSLCurve) node[sloped, midway, above, text width=2cm, align=center] {\small finite linear groups};
\draw[-latex] (SLCurve) -- (PSLCurve);
\node[propnode] (Haken) at (6,-1) [text width=1cm]{\small Haken};

\draw[-latex,thick] (SLCurve.east) -- (Haken) node[sloped, midway, above, align=center] {\small Culler-Shalen}; 
\draw[-latex,thick] (PSLCurve) -- (Haken) node[sloped, midway, above, align=center] {\small Boyer-Zhang }; 
\draw[-latex,thick] (ANI.east) -- (Haken) node[sloped, midway, below, align=center] {\small Culler-Shalen }; 
\draw[-,thick] (EVF) -- (6,2) node[midway, above]{\small Definition};
\draw[-latex,thick] (6,2) -- (Haken); 
\draw[-,thick] (Hass) -- (6,-4)  node[midway, below]{\small \cite[Theorem 3.1]{HassHomology}};
\draw[-latex,thick] (6,-4) -- (Haken); 
\end{tikzpicture}
\caption{The diagram summarizes the implication among the conditions in \cref{thm:mainhakentheorem} as well as the property of being Haken of a 3-manifold. Here, the (4) condition is referred to as having a large character variety. This terminology comes from the fact that a knot exterior is large if it contains a closed embedded essential surface. Finally, ``$\exists$ ANI character" denotes the last condition of \cref{thm:mainhakentheorem} in which ``ANI" strands for algebraic non-integral.} 
\end{figure}

\subsection{When $M$ has an embedded virtual fiber or $\pi_1(M)$ has an infinite dihedral quotient}
\begin{defn}
    An essential surface $\Sigma\subset M$ is a virtual fiber if there is a finite sheeted cover $p:M'\to M$ such that a component of $p^{-1}(\Sigma)$ is a fiber of a fibration of $M'$ over $S^1$.
    \end{defn}

\begin{lemma}\label{lem:EVF_Detection}
Let $M$ be a hyperbolic $\mathbb{Q}HS^3$ containing an embedded, essential virtual fiber $\Sigma$, and let $N$ be as above. Then $N$ is Haken and contains an embedded, essential virtual fiber.
\end{lemma}

We provide the first proof of \cref{lem:EVF_Detection} communicated to us by Alan Reid. 

\begin{proof}[The first proof of \cref{lem:EVF_Detection}]
The 3-manifold $M$ is a union of two twisted $I-$bundles (over a non-orientable surface) $A$ and $B$ along their boundary $\Sigma$ (\cite[Lemma 12.27]{Purcell}). The manifolds $A$ and $B$ are double covered by $\Sigma\times I$. Thus, there is an epimorphism $\pi_1(M)\twoheadrightarrow \Z/2\Z\star\Z/2\Z \cong D_\infty$. Since $D_\infty$ surjects infinitely many finite dihedral groups, $\pi_1(M)$ surjects infinitely many finite dihedral groups. Finite dihedral groups are subgroups of PSL$(2,\C)$ and so the group $\pi_1(M)$ has infinitely many non-conjugate $\PSL(2,\C)$ representations. As $\pi_1(N)$ has the same finite quotients as $\pi_1(M)$, $\pi_1(N)$ also has infinitely many distinct finite dihedral $\PSL(2,\C)$ representations. Thus, the $\PSL(2,\C)$ character variety of $N$ has a curve, and by \cite[Theorem 4.3]{BoyerZhang}, $N$ is Haken. 
\end{proof}

Now we present a second proof of \cref{lem:EVF_Detection} using the work of Liu \cite{LiuAlmostProfiniteRigidity} to obtain a stronger conclusion about the essential surface in $N$. 

\begin{proof}[The second proof of \cref{lem:EVF_Detection}]
     Similar to the previous proof, we start with an epimorphism $f:\pi_1(M)\twoheadrightarrow D_\infty$ to the infinite dihedral group with completion $\widehat{f}:\widehat{\pi_1(M)}\twoheadrightarrow\widehat{D}_\infty$. Recall that we have an isomorphism $\Phi:\widehat{\pi_1(N)}\to\widehat{\pi_1(M)}$. Let $\pi_1(M')$ be the index-two subgroup of $\pi_1(M)$ that surjects the index-two copy of $\Z$ in $D_\infty$. The restriction $f|_{\pi_1(M')}$ has finitely generated kernel since it is the fundamental group of the fiber surface of $M'$ over $S^1$ where the fiber is a lift of $\Sigma$. The intersection of $\pi_1(N)$ with $\Phi^{-1}(\widehat{\pi_1(M')})$ is an index-two subgroup $\pi(N')$ of $\pi_1(N)$. The restriction of $\Phi$ to $\widehat{\pi_1(N')}$ is an isomorphism $\Phi:\widehat{\pi_1(N')} \to \widehat{\pi_1(M')}$. Let $\Phi_*:\widehat{H}_1(N',\Z)_{\free} \to :\widehat{H}_1(M',\Z)_{\free}$ be the induced homomorphism. By \cref{thm:zhatregularity}, there exists a unit $\mu \in \widehat{\mathbb{Z}}^\times$ and an isomorphism 
     \[
     s: H_1(N',\Z)_{\free} \to H_1(M';\Z)_{\free}
     \]
     such that $\Phi_*= m_\mu \circ \widehat{s}$ where $m_\mu$ is left multiplication by $\mu$ in $\widehat{H}_1(M';\Z)_{\free}$. In particular, we can factor $\Phi_*$ as 
     \begin{equation}
         \begin{aligned}
     \label{eq:FactorPhiStar}
          \widehat{H}_1(N';Z)_{\free} = H_1(N';Z)_{\free} \otimes_{\Z} \widehat{\Z} \xrightarrow{s \otimes 1} 
          &H_1(M';Z)_{\free} \otimes_{\Z} \widehat{\Z} \\  
          \xrightarrow{1 \otimes \mu} & H_1(M';Z)_{\free} \otimes_{\Z} \widehat{\Z}=\widehat{H}_1(N';Z)_{\free} 
     \end{aligned}
     \end{equation}
     
     The map $\widehat{f} \circ \Phi: \pi_1(N') \to \widehat{\Z/2\Z*\Z/2\Z}$ factors as
     \[
     \pi_1(N') \hookrightarrow \widehat{\pi_1(N')} \xrightarrow{p}  \widehat{H}_1(N';Z)_{\free} \xrightarrow{\Phi_* =  m_\mu \circ \widehat{s}} \widehat{H}_1(M';Z)_{\free} \xrightarrow{\widehat{f}_*} \widehat{\Z}.
     \]
     Using \cref{eq:FactorPhiStar}, we can compute the image of $\pi_1(N')$ under $\widehat{f} \circ \Phi$ which is infinite cyclic
     \begin{equation}
     \begin{aligned}
     \widehat{f}(\Phi(\pi_1(N'))) 
     = &\widehat{f}_*(m_\mu(s(H_1(N';\Z)))) \\ 
     = &\widehat{f}_*(\mu H_1(M';\Z)) = \mu f(\pi_1(M')) \cong \mu \Z 
     \end{aligned}
     \end{equation}
     
     \noindent
     Thus, the image $\widehat{f}\circ \Phi(\pi_1(N))$ in $\widehat{\Z/2\Z\star\Z/2\Z}$ has an infinite cyclic subgroup of index 2. 
     
     Since $M'$ fibers with a fiber surface dual to $f|_{\pi_1(M')}$, the kernel of $f|_{\pi_1(M')}$ is finitely generated. By \cite[Corollary 6.2]{LiuAlmostProfiniteRigidity}, the kernel of $\widehat{f}\circ\Phi|_{\pi_1(N')}$ is also finitely generated. Since $N$ is a $\mathbb{Q}HS^3$, $(f\circ \Phi)(\pi_1(N))\not\cong\Z$, and so $(f\circ \Phi)(\pi_1(N))$ is infinite dihedral and $f\circ\Phi|_{\pi_1(N)}$ has finitely generated kernel. It follows that $\pi_1(N)$ acts non-trivially on the line, and is therefore Haken. By \cite[Theorem 11.8]{Hempel3mfds}, $N$ contains an embedded virtual fiber.  
\end{proof}
The method of the second proof of \cref{lem:EVF_Detection} yields a stronger result. That is, having an infinite dihedral quotient is detectable by finite quotients. Furthermore, containing an embedded virtual fiber is also detectable by finite quotient. Since Property $\mathcal{H}$ is equivalent to having an infinite dihedral quotient \cref{prop:EffectivePropertyH}, we have the following corollary of \cref{prop:EffectivePropertyH} and \cref{lem:EVF_Detection}.

\begin{cor}\label{cor:detecting_EVF_Dihedral_H}
    Let $(M,N;\Phi)$ be a profinite pair in the Haken setting. If $M$ contains an embedded virtual fiber, then $N$ also contains an embedded virtual fiber. If $\pi_1(M)$ has an infinite dihedral quotient, or equivalently Property $\mathcal{H}$, then $\pi_1(N)$ has an infinite dihedral quotient, or equivalently Property $\mathcal{H}$. In all cases, $N$ is also Haken.  
\end{cor}

\subsection{Property $\mathcal{H}$}\label{section:virtualfixedclasses}

In this section, we provide a direct proof to the fact that the profinite completion of $\pi_1(M)$ detects Property $\mathcal{H}$. \cref{lem:2ZGmodules} could be of independent interest. More precisely, we will prove the following.

\begin{lemma}\label{lem:mainthmpart4restated}
Let $(M,N;\Phi)$ be a profinite pair in the Haken setting. If $M$ satisfies Property $\mathcal{H}$, then $N$ must also satisfy Property $\mathcal{H}$. In particular, $N$ is Haken.
\end{lemma}

\cref{lem:mainthmpart4restated} is a consequence of 
\begin{lemma}\label{lem:2ZGmodules}
Let $M,N$ be profinitely equivalent finite-volume hyperbolic 3-manifolds and $\Phi:\widehat{\pi_1(N)} \to  \widehat{\pi_1(M)} $ be an isomorphism. Let $M'\to M$, $N'\to N$ be a pair of corresponding regular, finite-sheeted covers of $M$ and $N$ respectively with deck group $G$. Then $H_2(M';\Z)\cong H_2(N';\Z)$ as $\Z G$-modules
\end{lemma}

\noindent
We now proceed to prove \cref{lem:mainthmpart4restated} assuming \cref{lem:2ZGmodules}.
\begin{proof}[Proof of \cref{lem:mainthmpart4restated}] Let $N'\to N$ be the cover of $N$ corresponding to $M'\to M$. By \cref{lem:2ZGmodules}, $H_2(M';\Z)\cong H_2(N';\Z)$ as $\Z G-$modules. It follows, then, that there is a (non-zero) class $\beta\in H_2(N';\Z)$ such that $g_*\beta=\pm\beta$ for all $g\in G$. By Theorem 3.1 \cite{HassHomology}, $N$ contains an essential embedded surface which is 2-sided. 
\end{proof}

\begin{proof}[Proof of \cref{lem:2ZGmodules}]
The data of the setup is the following: an isomorphism $\Phi:\widehat{\pi_1(N)}\to\widehat{\pi_1(M)}$ and a (non-trivial) epimorphism $f:\pi_1(M)\to G$ for some finite group $G$. Set $\pi_1(M') = \ker f$ and $\displaystyle \pi_1(N')=\ker(\widehat{f}\circ \Phi)\cap\pi_1(N)$. The action of $\pi_1(M)$ on $\pi_1(M')$ by conjugation gives a homomorphism $\pi_1(M)\to \Aut(\pi_1(M'))$. In similar fashion, we obtain a compatible action of $\widehat{\pi_1(M)}$ on $\widehat{\pi_1(M')}$ giving a compatible homomorphism $\widehat{\pi_1(M)}\to \Aut(\widehat{\pi_1(M')})$. By \cite[Corollary 4.4.4 ]{RibesZalesskiiBook}, $ \Aut(\widehat{\pi_1(M')})$ is a profinite group because $\widehat{\pi_1(M)}$ is a finitely generated profinite group. Compatibility means there are canonical injections $\pi_1(M)\to\widehat{\pi_1(M)}$ and $\Aut(\pi_1(M'))\to \Aut(\widehat{\pi_1(M')})$ making the following diagram commute
\[
\begin{tikzcd}
  \pi_1(M) \arrow[r] \arrow[d]
    & \Aut(\pi_1(M')) \arrow[d] \\
  \widehat{\pi_1(M)} \arrow[r]
 &\Aut(\widehat{\pi_1(M')}) \end{tikzcd}
 \]
Any element of the group $G$ of covering automorphisms of $M'\to M$ will induce an automorphism of $\pi_1(M')$ which is conjugation by some element $g\in\pi_1(M)$. The element $g$ induces an automorphism of $\widehat{\pi_1(M')}$. Abusing notation, we denote by $\widehat{g}:\widehat{\pi_1(M')}\to\widehat{\pi_1(M')}$ the automorphism induced by $g$. 

We identify $\pi_1(M)$ with its image in $\widehat{\pi_1(M)}$ under the canonical embedding. Let $g \in \widehat{\pi_1(M)}$ and denote by $\varphi_g: \widehat{\pi_1(M')} \to \widehat{\pi_1(M')}$ be the automorphism induced by $g$. Since the subgroup $\Phi(\pi_1(N))$ is dense in $\widehat{\pi_1(M)}$, there exist a sequence $\{\alpha_n \} \subset \pi_1(N) $ such that $\Phi(\alpha_n) \to g$. We observe that
\[
\varphi_{\Phi(\alpha_n)} =\Phi \varphi_{\alpha_n} \Phi^{-1}
\]
where $\varphi_{\alpha_n}:\widehat{\pi_1(N')}\to\widehat{\pi_1(N')}$ is an automorphism induced by conjugation by $\alpha_n$. Therefore, the sequence $\Phi \varphi_{\alpha_n} \Phi^{-1}$ converges to $\varphi_g$.

Let $\widehat{D}(\pi_1(M'))<\widehat{\pi_1(M)}$ be the subgroup of $\widehat{\pi_1(M)}$ consisting of elements $g$ for which $\varphi_g:\widehat{\pi_1(M')}\to\widehat{\pi_1(M')}$ induces the identity isomorphism $\mathbf{1}:\widehat{H}_1({\pi_1(M')},\Z)_{\free}\to \widehat{H}_1({\pi_1(M')},\Z)_{\free}$. The group $\widehat{D}(\pi_1(M'))<\widehat{\pi_1(M)}$ is finite-index because it contains the finite-index subgroup of inner automorphisms. For large enough $n$, $\Phi\varphi_{\alpha_n}\Phi^{-1}$ and $\varphi_g$ will be in the same coset of $\widehat{D}(\pi_1(M'))<\widehat{\pi_1(M)}$ and so will induce the same map on first continuous profinite homology groups. Thus, for large enough $n$, we have the following commutative diagram
\[
\begin{tikzcd}
\widehat{H_1}(\pi_1(N'),\Z)_{\free}\arrow[r, "(\varphi_{\alpha_n})_*"] \arrow[d, "\Phi_*"] & \widehat{H_1}(\pi_1(N'),\Z)_{\free} \arrow[d, "\Phi_*"]\\
\widehat{H_1}(\pi_1(M'),\Z)_{\free}\arrow[r, "(\varphi_{g})_*"]  & \widehat{H_1}(\pi_1(M'),\Z)_{\free} 
\end{tikzcd}
\]
By \cref{thm:zhatregularity}, $\Phi_*=\mu\circ \widehat{h}$ where $\mu$ denotes the multiplication by a unit $\mu\in\widehat{\Z}^\times$ and $\widehat{h}$ is the profinite completion of a honest automorphism $H_1N'\to H_1M'$ as abstract groups. We can consider the module automorphism $1/\mu:\widehat{H_1}(\pi_1(M'),\Z)_{\free}\to\widehat{H_1}(\pi_1(M'),\Z)_{\free} $ which is multiplication by the inverse of $\mu$. Since $\widehat{\Z}$ is a commutative ring, the multiplication map commutes with other profinite automorphisms of $\widehat{H}_1(\pi_1(M');\Z)_{\free}$, and so we have the following commutative diagram

\[
\begin{tikzcd}
\widehat{H_1}(\pi_1(N'),\Z)_{\free}\arrow[r, "(\varphi_{\alpha_n})_*"] \arrow[d, "\Phi_*"] & \widehat{H_1}(\pi_1(N'),\Z)_{\free} \arrow[d, "\Phi_*"]\\
\widehat{H_1}(\pi_1(M'),\Z)_{\free}\arrow[r, "(\varphi_g)_*"] \arrow[d, "1/\mu"] & \widehat{H_1}(\pi_1(M'),\Z)_{\free} \arrow[d, "1/\mu"] \\
\widehat{H_1}(\pi_1(M'),\Z)_{\free}\arrow[r, "(\varphi_g)_*"] & \widehat{H_1}(\pi_1(M'),\Z)_{\free}
\end{tikzcd}
\]

By restriction, this diagram reduces to an analogous diagram at the discrete level showing that the conjugation $\alpha_n$ is a conjugate of $g$ by a fixed homomorphism obtained by restricting $1/\mu\circ \Phi_*$. 
\end{proof}

\subsection{When $M$ has a large $\SL_2$-character variety over $\mathbb{C}$}
For the next two sections, we use notation from \cite{BMRS1}. For a group $\Gamma$, and for a finite rational prime $p$, let 
\begin{align*}
    X(\Gamma,\C) &=\{\text{conjugacy classes of SL}(2,\C)\text{ representations of }\Gamma\}\\
    X(\Gamma,\overline{\Q_p}) &=\{\text{conjugacy classes of SL}(2,\overline{\Q_p})\text{ representations of }\Gamma\}\\
    X_b(\Gamma,\overline{\Q_p}) &=\{\text{conjugacy classes of bounded SL}(2,\overline{\Q_p})\text{ representations of }\Gamma\}\\
    X_c(\widehat{\Gamma},\overline{\Q_p}) &=\{\text{conjugacy classes of continuous SL}(2,\overline{\Q_p})\text{ representations of }\widehat{\Gamma}\}
\end{align*}

\begin{lemma}
\label{lem:0DetectCurve}
    Let $(M,N;\Phi)$ be a profinite pair in the Haken setting. If $\pi_1(M)$ has a curve in its $\SL(2,\C)$-character variety, then $\pi_1(N)$ also has a curve in its $\SL(2,\C)$-character variety. In particular, $N$ is also Haken. 
\end{lemma}
\begin{proof}
For a contradiction, assume that $N$ is non-Haken. Then, there are finitely many conjugacy classes of $\SL(2,\C)$ representations of $\pi_1(N)$. Otherwise, there exists a curve in the $\SL(2,\C)$ character variety of $\pi_1(N)$ which implies that $N$ is Haken by \cite[Theorem 2.2.1]{CullerShalen}. By \cite[Lemma 4.2]{BMRS1}, the number of conjugacy classes of bounded $\SL(2,\overline{\Q_p})$ representations of $\pi_1(N)$ is no more than the number of conjugacy classes of $\SL(2,\C)$ representations which is a finite number $n$. Thus, there at most $n$ such bounded representations for all primes $p$. By out assumption on $M$ there exists a curve $C$ in $\mathcal{X}(M)$, the $\SL(2,\C)$ character variety of $M$. Pick $n+1$ points on $C$ that correspond to $n+1$ pairwise non-conjugate representations $\phi_1,\dots,\phi_{n+1}:\pi_1(M) \to \SL(2,\C)$. By  \cite[Lemma 4.2]{BMRS1}, there is a prime $p$ and a field isomorphism $\theta:\C \to \overline{\Q_p}$ such that the $n+1$ representations $\theta_*(\phi_i)$ correspond to $n+1$ characters of bounded SL$(2,\overline{\Q_p})$ representations of $\pi_1(M)$. By \cite[Lemma 4.3]{BMRS1}, there is a natural bijection between conjugacy classes of bounded $\SL(2,\overline{\Q}_p)$-representations of any finitely generated group $\Lambda$ and conjugacy classes of continuous $\SL(2,\overline{\Q}_p)$-representations of $\widehat{\Lambda}$. Using this bijection, we obtain $n+1$ conjugacy classes of continuous representations of $\widehat{\pi_1(M)}$ and also of $\widehat{\pi_1(N)}$ by \cite{ns}. Using the bijection from \cite[Lemma 4.3]{BMRS1} again, we obtain $n+1$ conjugacy classes of bounded $\SL(2,\overline{\Q_p})$ representations of $\pi_1(N)$, a contradiction. Thus, $\pi_1(N)$ has infinitely many conjugacy classes of $\SL(2,\C)$ representations. It follows that there is a curve in the character variety of $\pi_1(N)$, and $N$ is therefore Haken by \cite[Theorem 2.2.1]{CullerShalen}.    
\end{proof}

\begin{remark}
    \cite[Theorem 11.1]{LiuArithmeticDetection} establishes a bijection between the $\SL(2,\Q^{ac})$ (where $\Q^{ac}$ is a fixed algebraic closure of $\Q$) characters of the fundamental groups of profinitely equivalent 3-manifolds. We do not use this theorem in \cref{lem:0DetectCurve} above, instead we rely on a direct counting argument. For examples of hyperbolic $\mathbb{Q}HS^3$ with a large $\SL(2,\mathbb{C})$-character variety, we refer the reader to \cref{section:examples}. 
\end{remark}

\subsection{When $M$ has a large $\SL_2$-character variety over the field with positive characteristic.}
In the case where $\F$ is an algebraically closed field of positive characteristic, and for a finitely generated group $\Gamma$ we follow
Garden--Tillmann \cite{ptilmanngarden}, Paoluzzi--Porti \cite{paoluzziporti} to define $\mathcal{X}(\Gamma,\F)=\{\rho:\Gamma\to \SL(2,\F)\}/\sim$ up to conjugation in $\SL(2,\F)$. The following lemma is standard.
\begin{lemma}\label{lem:pdetectcurve}
    If $\Gamma,\Delta$ are finitely generated, residually finite groups with $\Phi:\widehat{\Gamma}\to\widehat{\Delta}$ an isomorphism, then there is a $\Phi-$equivariant bijection $X(\Gamma,\F)\cong X(\Delta,\F)$.
\end{lemma}

\begin{proof}
Fix a group isomorphism $\Phi:\widehat{\Gamma}\to\widehat{\Delta}$. The image of a finitely generated group in $\SL(2,\F)$ is finite. Thus, every representation $\rho:\Gamma\to \SL(2,\F)$ has a unique extension $\widehat{\rho}:\widehat{\Gamma}\to \SL(2,\F)$. The homomorphism $\widehat{\rho}\circ \Phi^{-1}$ restricted to $\Delta$ has the same image as $\rho$ in $\SL(2,\F)$. For the reverse direction, for a homomorphism $\varphi:\Delta\to \SL(2,\F)$, $\widehat{\varphi}\circ\Phi$ restricted to $\Gamma$ is a homomorphism with the same image as $\varphi$. Since the conjugation action of $\SL(2,\F)$ preserves the isomorphism type of the groups, it follows that $X(\Gamma,\F)\cong X(\Delta,\F)$ via a $\Phi-$equivariant bijection.
\end{proof}

In the case of character varieties of 3-manifolds over fields with positive characteristic, we record that finite quotients detect embedded surfaces. 
\begin{lemma}\label{lem:pcurvehaken}
If $M$ is a Haken hyperbolic $\Q$H$S^3$ with a curve $C$ in $\mathcal{X}(\pi_1(M),\F)$ for an algebraically closed field $\F$ of positive characteristic, then $\mathcal{X}(\pi_1(N),\F)$ also contains a curve. In particular, $N$ is Haken. 
\end{lemma}

\begin{proof}
By \cref{lem:pdetectcurve}, there is a curve of characters in $\mathcal{X}(\pi_1(N),\F)$ since $X(\pi_1(N),\F)\cong X(\pi_1(M),\F)$. Applying \cite[Theorem 24]{ptilmanngarden} and \cite[Proposition 2.3.1]{CullerShalen}, we conclude that $N$ is Haken. 
\end{proof}

\subsection{The case of non-integral trace}
In this subsection, we prove \cref{thm:mainhakentheorem} (5).  
\begin{prop}\label{keyintegralcountprop}
Let $\Gamma$ be a group with $|X(\Gamma,\C)|<\infty$. Every $\SL(2,\C)$ representation of $\Gamma$ has integral trace iff $|X(\Gamma,\C)|=|X_b(\Gamma,\overline{\Q_p})|$ for all finite rational primes $p$.
\end{prop}
\begin{proof}
If every representation of $\Gamma$ has integral traces, then by the proof of \cite[Lemma 4.5 ]{BMRS1}, $|X(\Gamma,\overline{\Q_p})|=|X_b(\Gamma,\overline{\Q_p})|$ for all finite rational primes $p$. By fixing an identification between $\C$ and $\overline{\Q_p}$, $|X(\Gamma,\overline{\Q_p})|=|X(\Gamma,\C)|$ for all finite rational primes.
\medbreak\noindent For the converse, if $\Gamma$ has a representation $\rho:\Gamma\to \SL(2,\C)$ with non-integral trace, then for some finite rational prime $p$, there is a representation $\phi:\Gamma\to\,$SL$(2,\overline{\Q_p})$ that is not bounded. To see this, let $\alpha$ be the non-integral trace of some element $\rho(\gamma)$ for some $\gamma\in\Gamma$. For some prime $\frak p$ of the trace field $k=k_{\rho(\Gamma)}$, which is in the denominator of $\rho(\gamma)$, the valuation $\nu_\frak{p}(\alpha)>1$. By completing the field $k$ at $\frak{p}$, we get a representation $\rho_\frak{p}:\Gamma\to \,$SL$(2,k_{\nu_\frak{p}})$ which is not bounded. The prime $\frak{p}$ lies above a finite rational prime $p$, and therefore $k_{\nu_\frak{p}}$ is a finite extension of $\Q_p$. Thus, the representation $\rho_\frak{p}$ is an unbounded representation $\Gamma\to \,$SL$(2,\overline{\Q_p})$. For this prime $p$, $|X_b(\Gamma,\overline{\Q_p}|<|X(\Gamma,\C)|$. 
\end{proof}

\begin{lemma}\label{lem:integralcount}
    Let $\Gamma$ be a group with finitely many conjugacy classes of $\SL(2,\C)$ representations. For $\Delta$ with $\widehat{\Delta}\cong\widehat{\Gamma}$, $\Gamma$ has an $\SL(2,\C)$ representation with non-integral trace iff $\Delta$ has an $\SL(2,\C)$ representation with non-integral trace. 
\end{lemma}
\begin{proof}
    If $\Gamma$ has an SL$(2,\C)$ representation with non-integral trace, there is a rational prime $p$ for which $|X_b(\Gamma,\overline{\Q_p})|<|X(\Gamma,\C)|$ by \cref{keyintegralcountprop}. By Lemma 4.3 \cite{BMRS1}, for all finite rational primes $l$ ,$$|X_b(\Delta,\overline{\Q_l})|=|X_c(\widehat{\Delta},\overline{\Q_l})|=|X_c(\widehat{\Gamma},\overline{\Q_l})|=|X_b(\Gamma,\overline{\Q_l})|$$
    since $\widehat{\Gamma}\cong\widehat{\Delta}$. By Proposition 4.1 \cite{BMRS1}, $|X(\Gamma,\C)|=|X(\Delta,\C)|$. Therefore, $|X_b(\Delta,\overline{\Q_p})|<|X(\Delta,\C)|$, and by \cref{keyintegralcountprop}, $\Delta$ has an SL$(2,\C)$ representation with non-integral trace. The same argument works for the other direction. 
\end{proof}

\begin{lemma}\label{lem:noninttraceLemma}
    Let $M$ be a hyperbolic $\Q$H$S^3$ with $0$-dimensional $\SL(2,\C)-$character variety, and a representation $\rho:\pi_1(M)\to \SL(2,\C)$ with non-integral trace. Then $N$ is Haken. 
\end{lemma}
\begin{proof}
    Since $M$ has a $0$-dimensional character variety, every SL$(2,\C)$ representation of $\pi_1(M)$ has algebraic traces. By the hypothesis, $\pi_1(M)$ has an SL$(2,\C)$ representation with non-integral trace, and therefore $\pi_1(N)$ has an algebraic SL$(2,\C)$ representation with non-integral trace by \cref{lem:integralcount} above. The 3-manifold $N$ is Haken by \cite[Theorem 5.2.7]{MaclachlanReidBook} (see also \cite[Theorem 6.5]{Bass}). This concludes the proof of \cref{lem:noninttraceLemma}. 
\end{proof}

\begin{proof}[Proof of \cref{thm:mainhakentheorem}]
    \cref{prop:EffectivePropertyH} shows that having infinite dihedral quotient is equivalent to Property $\mathcal{H}$ for $\Q$H$S^3$. \cref{cor:detecting_EVF_Dihedral_H} shows that if $M$ satisfies either (1), (2) or (3) then $N$ must also satisfies either (1), (2) or (3) respectively. The fact that $M$ has a large character variety over $\mathbb{C}$ or over $\overline{\mathbb{F}}_p$ implies that $N$ has a large character variety over $\mathbb{C}$ or over $\overline{\mathbb{F}}_p$, respectively is proven in \cref{lem:0DetectCurve} and \cref{lem:pdetectcurve}.
    This concludes the proof of \cref{thm:mainhakentheorem}.
\end{proof}

\section{Closed hyperbolic $\mathbb{Q}HS^3$s with Haken profinite genus}\label{section:examples}

In this section, we give examples of closed hyperbolic $\mathbb{Q}HS^3$s that satisfy criteria in \cref{thm:mainhakentheorem}. In particular, the profinite genus among 3-manifold groups of these examples consists only of Haken 3-manifolds. 

\subsection{Calegari--Dunfield tower of hyperbolic $\mathbb{Q}HS^3$}

In \cite{CalegariDunfield}, F. Calegari and Dunfield constructed a tower of arithmetic $\Q HS^3$
\begin{equation*}
    N_0 \leftarrow N_1 \leftarrow N_2 \leftarrow N_3 \leftarrow \dots
\end{equation*}
such that $\displaystyle \bigcap_{n=0}^\infty N_n = \{1\}$, and remarkably each $N_n$ is a Haken hyperbolic $\Q HS^3$, see \cite[Theorem 1.4]{CalegariDunfield} also \cite{BostonEllenberg}. The lifts to $\SL(2,\C)$ of the discrete faithful representations of the $\pi_1(N_n)$ have integral traces. However, Calegari-Dunfield show that each of the manifolds $N_n$ are Haken. See Section 2.10 \cite{CalegariDunfield} for details which we use in the proof below.

\begin{lemma}
    Let $L_n$ be a tower of hyperbolic  $\Q HS^3$ profinitely equivalent to the Calegari-Dunfield tower (Theorem 1.4 \cite{CalegariDunfield}). The manifolds $L_n$ are Haken for all $n$. 
    \begin{proof}
        By \cite[Section 2.10]{CalegariDunfield}, $N_n$ is a finite-sheeted regular cover of an orbifold $M_1$ (in the notation of \cite[Section 2.10]{CalegariDunfield}) with underlying space $\R P^3\#\R P^3\# L(4,1)$, and hence there exists an epimorphism $f:\pi_1(M_1) \to \Z/2\Z\star\Z/2\Z$. The image of $f(\pi_1(N_n))$ in $\Z/2\Z\star\Z/2\Z$ has no non-trivial homomorphism to $\Z$ since $N_n$ is a $\Q HS^3$. Therefore, $f(\pi_1(N_n))$ is isomorphic to an infinite dihedral group. By \cref{lem:EVF_Detection}, every manifold profinitely equivalent to $N_n$ is Haken. 
    \end{proof}
\end{lemma}

\noindent We suspect that the kernel of $f:\pi_1(N_n) \to \Z/2\Z\star\Z/2\Z$ is not finitely-generated and that the dual surface to $f$ corresponds to a separating totally geodesic surface in $N_n$.  


\subsection{$\Q HS^3$ with large $\SL_2$-character variety over $\C$}

In a discussion \cite{RubermanOverflow}, D. Ruberman gives the following construction of hyperbolic rational homology spheres with positive-dimensional $\SL_2(\C)$ character varieties constructed using invertible homology cobordisms. Let $K_1,K_2$ be non-trivial knots in $S^3$. The rational homology sphere $S(K_1,K_2)$ called the splice of $K_1$ and $K_2$ is obtained by gluing one knot complement meridian to the other knot complement longitude. By \cite[Theorem 2.6]{ruberman}, there is a hyperbolic rational homology sphere $M$ with an invertible homology cobordism to $S(K_1,K_2)$. This means that there exists two 4-manifolds $W$ and $W'$ such that $\partial W = \partial W' = M \sqcup S(K_1,K_2)$ and furthermore $W \cup_{M} W' \cong S(K_1,K_2) \times [0,1]$. The collapsing map of $S(K_1,K_2) \times [0,1]$ to $S(K_1,K_2)$ gives a degree one map $M\to S(K_1,K_2)$ which induces an epimorphism of the fundamental groups, see also \cite[Corollary 2.7]{RubermanOverflow}. Via this epimorphism, the character variety of $\pi_1(S(K_1,K_2))$ is included into the character variety of $\pi_1(M).$ In \cite{ruberman}, Ruberman proves that for any closed orientable 3-manifold $L$ there exists a closed orientable hyperbolic 3-manifold $M$ and an invertible homology cobordism from $M$ to $L$ \cite[Theorem 2.6]{ruberman}. For $L = S(K_1,K_2)$, $M$ is a hyperbolic $\Q HS^3$.

It remains to see that there exists $S(K_1,K_2)$ with large $\SL_2(\C)$- character variety. In fact, every spliced sum $S(K_1,K_2)$ of two nontrival knots $K_1$ and $K_2$ in $S^3$ has the property that every component of the $\SL_2(\C)$-character variety has positive dimension \cite[Corollary 3.3]{BodenCurtis2008}. 

For some explicit examples of computation of the character variety of a spliced sum of knots, see \cite[Section 4, 4.3]{KN}. Following Kitano--Nozaki \cite{KN}, we can choose $K_1,K_2$ to both be isotopic to the figure 8 knot, and then $\pi_1(S(K_1,K_2))$ will have a curve in its character variety. By the argument in the proceeding paragraphs, any hyperbolic rational homology sphere $M$ invertibly homology cobordant to $S(K_1,K_2)$ will have a curve in its character variety, and any closed 3-manifold $N$ profinitely equivalent to $M$ must be Haken by \cref{lem:0DetectCurve}. 

\subsection{$\Q HS^3$ with large $\SL_2$-character variety over $\overline{\F}_p$}

In Section 6.4 \cite{ptilmanngarden}, Garden-Tillmann give a remarkable example of a hyperbolic $\Q$H$S^3$ {\tt m188(2,3)} (a Dehn-filling on a census \cite{SnapPy} manifold) with a curve in its characteristic 2 character variety and no curves in the character variety for any other characteristics including zero. A 3-manifold profinitely equivalent to {\tt m188(2,3)} is Haken by \cref{lem:pcurvehaken}. On further inspection, the manifold {\tt m188(2,3)} has a non-trivial homomorphism from its fundamental group to $\Z/2\Z\star\Z/2\Z$, and therefore it has Property $\mathcal{H}$.

\subsection{Haken $\Q HS^3$ from the Hodgson--Weeks census}

In this section, we describe some computer experiments with manifolds from the Hodgson--Weeks census \cite{HodgsonWeeksCensus}. These computations were performed using {\tt Regina} \cite{regina}, {\tt SnapPy} \cite{SnapPy}, {\tt SageMath} \cite{sagemath}, the {\tt Docker} image of SnapPy in SageMath developed by Dunfield as well as the {\tt SnapPyNT} extension program by Nicholas Rouse \cite{NickSnapPyNT}. The code and data can be found among the ancillary files of the arxiv post of this paper.  

\medskip\noindent{\bf Filtering Haken $\Q HS^3$.} There have been effort made toward determining Haken manifolds in the Hodgson--Weeks census in connection to the Virtual Haken Conjecture \cites{DunfieldThurstonHakenExperiment,DunfieldHakenHWCensus} and to effective algorithms in computational 3-manifold \cites{burton2025computingclosedessentialsurfaces, BurtonTillmann2025ComputingHaken}. Unable to obtain the data mentioned in \cite[Section 6]{burton2025computingclosedessentialsurfaces}, we re-compute this data using Regina \cite{regina}, see {\tt HakenCensus.rga} in the ancillary files. Roughly, if the number of tetrahedra in the triangulation of the manifold contained in the census is less than 20, then we test if the manifold is Haken or not. From \numCensus examples of orientable examples in the Hodgson Weeks census, we found at least \numHaken Haken 3-manifolds. From this data, we found \numHakenQHS examples that are $\Q HS^3$. We refer to the list of these \numHakenQHS Haken $\Q HS^3$ as {\tt Haken\_QHS\_List}. 

For each manifold $M$ in {\tt Haken\_QHS\_List}, we determine if $\pi_1(M)$ admits the infinite dihedral group as a quotient and if $M$ has a curve in its $\SL_2(\C)$-character variety using the following procedures. 

\medskip\noindent{\bf Infinite dihedral quotient.} To test for the infinite dihedral quotient $D_\infty$, we test if $M$ admits a double cover with positive first betti number. If $M$ has no double cover or all double covers are $\Q HS^3$, then $\pi_1(M)$ does not have any infinite dihedral quotient. We rule out \numNoDihedralQuotient $\Q HS^3$ from {\tt Haken\_QHS\_List} from having $D_\infty$ as a quotient by looking at double covers of $M$. If $\pi_1(M)$ pass the previously described double cover test, we search for an epimorphism from $\pi_1(M)$ to $D_\infty$ taking advantage of the fact that most of these fundamental groups are 2-generated. If $\pi_1(M)$ is 2-generated, say by $a$ and $b$, then we test the following assignment
\[
\begin{cases}
    a &\mapsto x \\
    b &\mapsto y
\end{cases},
\begin{cases}
    a &\mapsto x \\
    b &\mapsto xy
\end{cases},
\begin{cases}
    a &\mapsto xy \\
    b &\mapsto x
\end{cases},
\begin{cases}
    a &\mapsto x \\
    b &\mapsto y
\end{cases},
\]
where $x,y$ come from the follow presentation $D_\infty = \langle x,y\mid x^2,y^2 \rangle$. Up conjugation in $D_\infty$ and exchanging $x$ and $y$, any epimorphism to $D_\infty$ must take the form of one of the homomorphism above. When $\pi_1(M)$ is 3-generated, we test homomorphisms obtained from short words in $D_\infty$ and from solving a certain system of equations over $\mathbb{Z}$ coming from identifying $D_\infty$ with a subgroup of the affine group of the line generated by
\[
t \mapsto -t \quad\text{and}\quad t\mapsto-t + 1
\]
As the final result, we found \numDihedralQuotient out of \numHakenQHS examples in {\tt Haken\_QHS\_List} whose fundamental group admits the infinite dihedral group as a quotient. In other words, the double cover test found all examples of 3-manifolds from {\tt Haken\_QHS\_List} whose fundamental group has no infinite dihedral quotients.  

\begin{remark}
We could choose to rule out $D_\infty$ quotient by ruling out finite dihedral $D_{2p}$ of order $2p$ for primes $p$ up to some threshold. This test will not rule out anymore examples. If all double covers of $M$ are $\Q HS^3$, then $\pi_1(M)$ cannot admit a finite dihedral quotient $D_{2p}$ for some prime $p$ larger than the size of $|H_1(M';\Z)|$ where $M'$ range over all double covers of $M$. We implement the double cover test since it is more efficient.    
\end{remark}
\medskip\noindent{\bf Large character variety.} To test for a curve in the $\SL_2(\C)$-character variety, we use the routine 
comes with using {\tt SnapPy} in {\tt sage} \cites{SnapPy,sagemath}. 
\begin{verbatim}
    import snappy
    M = snappy.Manifold("m004")
    G = M.fundamental_group()
    I = G.character_variety_vars_and_polys("as_ideals")
    I.dimension()
\end{verbatim}
which returns the defining ideal of the character variety of $\pi_1(M)$ in terms of the trace functions of $\{a,b,ab\}$ for $2$-generated group and of $\{a,b,c,ab,ac,bc,abc\}$ for 3-generated group. The computation of the character variety and of the dimension of an ideal is generally expensive. We only managed to obtain the defining ideals for \numWithIdeal examples. Of those \numWithIdeal examples, we computed the dimension of the character variety for \numWithDimensionComputed examples. We determine that there are \numWithDimensionOne examples with at least one component of dimension one and \numWithDimensionZero examples with zero-dimensional character variety.

\medskip\noindent{\bf Non-integral traces.} We now verify that certain hyperbolic $\Q HS^3$s in the Hodgson--Weeks closed orientable census have non-integral traces and therefore have a profinite genus consisting of Haken 3-manifolds by \cref{thm:mainhakentheorem}.
\begin{theorem}\label{hakenincensus}
    Let $M$ be one of the following census hyperbolic 3-manifolds: 
    \begin{align*}
        &{\tt m015(8, 1), m019(3, 4), m026(-5,2), m036(-4, 3), m040(-4, 3), m007(5,3)} \\
       & {\tt m037(4, 3), m034(5,2), 
        m082(1, 3),  
        m145(1,3), m070(-3,2), m067(-3,2).}
    \end{align*}
If $N$ is profinitely equivalent to $M$, then $N$ is Haken.
\begin{proof}
    All of these manifolds are known to have discrete faithful representations with non-integral traces. We check this using a script using the SnapPyNT \cite{NickSnapPyNT} extension program developed by Nicholas Rouse. 
\end{proof}
\end{theorem}
We highlight some other interesting examples that have appeared in the literature. 

\begin{example} The manifold {\tt m015(8,1)} is a hyperbolic Haken $\Q HS^3$ with zero-dimensional character variety. The manifold can also be obtained from doing $10/1$-Dehn surgery on the knot $5_2$. The character corresponds to the discrete faithful representation is algebraic non-integral. The details of this example are written down in \cite[Example 5.2.8 (2)]{MaclachlanReidBook}. In particular, the square of the trace of the image of a meridian element of $\pi_1(S^3\setminus 5_2)$ in the fundamental group of the $10/1$ Dehn surgery is non-integral with a minimal polynomial $2x^4-17x^3+46x^2-40x+8$.   
\end{example}

\begin{example}
The manifold {\tt m019(3,4)} is a hyperbolic Haken $\Q HS^3$ with zero-dimensional $\SL_2$-character variety over $\mathbb{C}$. Garden--Tillmann observed that {\tt m019(3,4)} contains no curve of irreducible $\SL_2$-characters over $\overline{\mathbb{F}}_p$ for all positive characteristics \cite[Section 6.3]{ptilmanngarden}. Thus the essential surface in {\tt m019(3,4)} is not detected in any characteristics. The first homology $H_1(M) \cong \mathbb{Z}/40 \mathbb{Z}$ is finite cyclic, and so {\tt m019(3,4)} has no infinite dihedral quotient. Nevertheless, {\tt m019(3,4)} has $\SL_2(\C)$-character that is algebraic non-integral. Thus, it has Haken profinite genus.       
\end{example}

\begin{example}
    The manifold {\tt m140(4,1)} is an arithmetic hyperbolic Haken $\Q HS^3$ with zero-dimensional $\SL_2(\C)$-character variety. It also admits an infinite dihedral quotient and hence there exists a curve of $\PSL_2(\C)$-character variety for {\tt m140(4,1)}. We could further determine that {\tt m140(4,1)} contains an embedded virtual fiber. This is done using Regina to cut along normal surfaces and check whether each connected component has fundamental group isomorphic to the fundamental group of a closed non-orientable surface \cite{regina}.  
\end{example}

\begin{example}
    Another remarkable example is the manifold {\tt m188(2,3)} which is a hyperbolic Haken $\Q HS^3$. The $\SL_2$-character variety of this manifold has a curve of irreducible characters only in characteristic zero and no curve of irreducible characters in any other characteristics including zero \cite[Section 6.4]{ptilmanngarden}. We also observe that the fundamental group of {\tt m188(2,3)} admits an infinite dihedral quotient. 
\end{example}

\section{Conclusion: Remarks and questions on the general case}\label{conclusion}

The argument in \cref{lem:EVF_Detection} suggests the following question whose affirmative answer would imply an affirmative answer to \cref{hakenquestion}. 

\begin{question}\label{surjfree}
Let $M,N$ be profinitely equivalent (finite-volume) hyperbolic 3-manifolds and let $\varphi:\pi_1(M)\twoheadrightarrow F_r$ be an epimorphism to a non-abelian free group. Fix a group isomorphism $\Phi:\widehat{\pi_1(M)}\to \widehat{\pi_1(N)}$. Is $\varphi\circ \Phi^{-1}(\pi_1(N))$ free?
\end{question}

\noindent It is immediate that this group $\varphi\circ \Phi^{-1}(\pi_1(N))$ is a finitely generated dense subgroup of the free profinite group. A positive answer to this question will imply a positive answer to \cref{hakenquestion} as the following proposition shows.

\begin{prop}
Let $M$ be a Haken, hyperbolic $\Q HS^3$ and let $N$ be a 3-manifold with $\widehat{\pi_1(M)}\cong\widehat{\pi_1(N)}$. If there is a positive answer to \cref{surjfree}, then $N$ is Haken. 
\end{prop}

\begin{proof}
By \cref{thm:mainhakentheorem}, we can assume that $M$ contains an embedded quasiFuchsian surface $\Sigma$. Then $M=M_1\cup_\Sigma M_2$, and $\pi_1(M)=\pi_1(M_1)*_{\pi_1(\Sigma)}\pi_1(M_2)$. The image of $\pi_1(M_i)$ for $i=1,2$ in some finite quotient $Q$ are (non-trivial) finite groups $A,B$ with a common subgroup $C$ which is the image of $\pi_1(\Sigma)$. Since $\pi_1(M)$ is LERF, we may assume that $C$ is a proper subgroup of $A$ and $B$. Thus, there is an epimorphism to a virtually free group $\rho:\pi_1(M)\twoheadrightarrow A*_CB$. There is therefore a finite-index subgroup $\pi_1(M')<\pi_1(M)$ for which $\rho(\pi_1(M'))$ is a free group. A positive answer to \cref{surjfree} will mean that the image of $\widehat{\rho}(\overline{\pi_1(M)}\cap \pi_1(N))$ will be a free group, and as such, the group $\widehat{\rho}(\pi_1(N))$ will be virtually free. It then follows that $\pi_1(N)$ will act non-trivially on the Bass-Serre tree corresponding to the graph of groups decomposition of $\widehat{\rho}(\pi_1(N))$ with all vertex groups finite and having bounded order (Theorem 7.3 \cite{ScottWall}), and so $N$ will be Haken. 
\end{proof}

Furthermore, in the general case, the observation of \cref{lem:2ZGmodules} allows us to consider what happens for $\Sigma\subset M$ a quasiFuchsian essential embedded surface. We can pass to a finite-sheeted regular cover $M'\to M$ where all the lifts of $\Sigma$ are non-separating and necessarily disjoint. Fixing one such lift $\Sigma'$, we get a collection of lifts $g\Sigma$ for $g\in\pi_1(M)/\pi_1(M')$ and non-trivial homology classes $[g\Sigma]\in H_2(M';\Z)$. We observe that this collection of homology classes is linearly dependent because $\sum_{g\in G}[g\Sigma]\in H_2(M';\Z)$ is $G=\pi_1(M)/\pi_1(M')$-invariant and $M$ is a $\Q $H$S^3$, so $\sum_g[g\Sigma]=0$. For $N$ profinitely equivalent to $M$, we have a corresponding finite-sheeted cover $N'$ with a corresponding (via a group ring isomorphism in \cref{lem:2ZGmodules}) class $[\alpha]\in H_2(N';\Z)$ such that $\sum_{g\in G}[g\alpha]=0$. We are unable to directly guarantee from integral homology classes that, for example, the least area surface representatives for $[\alpha]$ and $[g\alpha]$ will be disjoint for all $g\in G=\pi_1(M)/\pi_1(M')=\pi_1(N)/\pi_1(N')$. 

\medbreak\noindent For the next question, we recall the following definition:
\begin{defn}
    A tame knot $K\subset S^3$ is {\bf small} if its knot exterior $E_K=S^3\setminus n(K)$ (a compact 3-manifold with toroidal boundary components) has no closed embedded essential surface.
\end{defn}
\begin{question}[Youheng Yao]
    If $K$ is a small knot and $\widehat{\pi_1}(E_K)\cong\widehat{\pi_1}(E_{K'})$ for some other knot $K'\subset S^3$, is $K'$ small? 
\end{question}
\noindent One can show \cite[Theorem 1.3]{CY} that there is a profinite alignment of strongly detected boundary slopes of $K$ and $K'$ and that Dehn surgery along all but the finitely many boundary slopes of the knot $K$ will yield non-Haken 3-manifolds. More needs to be said to guarantee that there is no closed embedded essential surface in $E_{K'}$. 

\section{Appendix: Property $\mathcal{H}$ and aspherical mod 2 homology 3-spheres}

We wish to record here that there are lots of examples of hyperbolic $\Q $H$S^3$s where there is no fixed class up to sign in a finite-sheeted regular cover. These examples are detailed below. In view of the fact that the virtual Betti number of these rational homology spheres is infinite (a consequence of the Virtual Special Theorem \cite{AgolHaken}) the following observations about the set of first Betti numbers of regular covers for a class of 3-manifolds indicate that there might be some restrictions on the sets of possible Betti numbers of finite regular covers for rational homology 3-spheres.

\begin{theorem}\label{mainthm}
    Let $M$ be an aspherical integer homology $3-$sphere and let $M'$ be a finite-sheeted regular cover of $M$. The first Betti number $b_1(M')\ne 1,2,3$. 
\end{theorem}

\begin{theorem}\label{secondthm}
    Let $M$ be an aspherical mod 2 homology $3-$sphere and let $M'$ be a finite-sheeted regular cover of $M$. The first Betti number $b_1(M')\ne 1$. 
\end{theorem}
We include a proof of a fact that is a well known consequence of the fact that every Sol 3-manifold is an Anosov semibundle (Proposition 12.7.6 \cite{Martelli}).
\begin{cor}\label{SOLVnomod2}
    There is no mod 2 homology 3-sphere that admits Sol geometry. 
    \begin{proof}
    If $M$ is a mod 2 homology 3-sphere with Sol geometry, as noted in the proof of \cref{structuralreduction}, we can pass to a finite-sheeted regular cover $M'\to M$ which is an Anosov torus bundle over $S^1$. However, we can check that $b_1(M')=1$, and that contradicts \cref{secondthm}.
    \end{proof}
\end{cor}

\begin{lemma}\label{infrestriction}
    Let $M$ be an aspherical 3-manifold with $b_1(M)=0$, and let $M'$ be a finite regular cover of $M$. Then $H^1(\pi_1(M'),\Z)^G=0$ where $G=\pi_1(M)/\pi_1(M')$. 
    \begin{proof}
    We consider the inflation-restriction exact sequence in group cohomology (see Proposition 3.3.17 \cite{GilleSzamuelyCSA} for example)
    $$0\to H^1(G,\Z)\to H^1(\pi_1(M),\Z)\to H^1(\pi_1(M'),\Z)^G\to H^2(G,\Z)\to H^2(\pi_1(M),\Z)$$
    Since $H^1(\pi_1(M),\Z)=0,$ then the torsion-free $H^1(\pi_1(M'),\Z)^G$ injects into $H^2(G,\Z)$ which is torsion because $G$ is finite (see Corollary 1.32 \cite{milneCFT} for example). Thus, $H^1(\pi_1(M'),\Z)^G=0$. 
    \end{proof}
\end{lemma}

\begin{lemma}\label{squares}
    Let $M$ be an aspherical mod 2 homology 3-sphere and let $M'$ be a finite regular cover of $M$. Then the $G=\pi_1(M)/\pi_1(M')$ action on $H^1(M';\Z)$ has no fixed class up-to-sign. 
    \begin{proof}
    The quotient of $\pi_1(M)$ by the subgroup $\pi_1(M)^{(2)}$ generated by squares is an elementary abelian 2-group. Since $M$ is assumed to be a mod 2 homology 3-sphere, $\pi_1(M)/\pi_1(M)^{(2)}$ is trivial. If there is a non-trivial class $\alpha\in H^1(M';\Z)$ that is fixed up to sign by $G$, then $G^{(2)}$ fixes $\alpha$. In this case however $G=G^{(2)}$, and this contradicts \cref{infrestriction}. 
    \end{proof}
\end{lemma}
\noindent With this, we may now complete the proof of \cref{mainthm} and \cref{secondthm}. 
\begin{proof}{(\cref{mainthm} \& \cref{secondthm})}
    Assuming $M'$ has $b_1=1$ and $M$ is a mod 2 homology sphere, then the $G=\pi_1(M)/\pi_1(M')$ action on $H^1(M';\Z)$ is fixed up to sign, contradicting \cref{squares}.  If $b_1(M')=n$, then the action on homology gives a homomorphism $\pi_1(M)\to \GL(n,\Z)$ with no 1-dimensional eigenspaces. In particular, the image of $\pi_1(M)$ in $\GL(n,\Z)$ must be non-trivial and finite. When $M$ is an integer homology sphere, $\pi_1(M)$ is perfect, and so the image of $\pi_1(M)$ is a finite perfect non-trivial subgroup of $\GL(n,\Z)$. For $n=2$ and $n=3$, there is no such subgroup (see the table on pg. 170 \cite{TaharaGL3Z}).
\end{proof}

\subsection{An example of a hyperbolic $\Z HS^3$ with finite-sheeted $b_1=4$ regular covers} We can perform Dehn surgery on the one-cusped manifold denoted by {\tt v3541} in the census of cusped 3-manifolds \cite{CHWeeks}. The manifold {\tt v3541(5,1)} has trivial integral homology and two $A_5$ covers with $b_1=4$. We can confirm this with the following scripts using SnapPy \cite{SnapPy} and Magma \cite{Magma}. First, in SnapPy \cite{SnapPy}
\medskip
\begin{verbatim}
    M=Manifold(`v3541')
    M.dehn_fill((5,1))
    M.homology()
    M.fundamental_group().magma_string()
\end{verbatim}
\medskip\noindent
Then in Magma \cite{Magma}
\begin{verbatim}
    G<a,b,c>:=Group<a,b,c|b*b*c*b^-1*a^-1*a^-1*c,
    a*c*a^-1*b^-1*c^-1*a*c*a^-1*b^-1*c^-1*a*c*b*b,
    a*c*c*c*c*c*b^-1*a*b*b>;
    AQInvariants(G);
    H:=LowIndexNormalSubgroups(G,60);
    for a in H do
        AQInvariants(a`Group);
    end for;
\end{verbatim}

\bibliography{main}

\begin{bibdiv}
\begin{biblist}

\bib{AgolHaken}{article}{
      author={Agol, I.},
       title={{The virtual Haken conjecture (with an appendix by I. Agol, D. Groves, and J. Manning)}},
        date={2013},
     journal={Documenta Math.},
      volume={{\bf 18}},
}

\bib{Bass}{article}{
      author={Bass, H.},
       title={{Groups of integral representation type}},
        date={1980},
     journal={Pacific J. Math.},
      volume={{\bf 86(1)}},
       pages={15–51},
}

\bib{regina}{misc}{
      author={Burton, Benjamin~A.},
      author={Budney, Ryan},
      author={Pettersson, William},
      author={others},
       title={Regina: Software for low-dimensional topology},
         how={{\tt http://\allowbreak regina-normal.\allowbreak github.\allowbreak io/}},
        date={1999--2025},
}

\bib{BodenCurtis2008}{article}{
      author={Boden, Hans~U.},
      author={Curtis, Cynthia~L.},
       title={Splicing and the {${\rm SL}_2(\Bbb C)$} {C}asson invariant},
        date={2008},
        ISSN={0002-9939,1088-6826},
     journal={Proc. Amer. Math. Soc.},
      volume={136},
      number={7},
       pages={2615\ndash 2623},
         url={https://doi.org/10.1090/S0002-9939-08-09233-2},
}

\bib{Magma}{article}{
      author={Bosma, W.},
      author={Cannon, J.},
      author={Playoust, C.},
       title={The {M}agma algebra system. {I}. {T}he user language},
        date={1997},
        ISSN={0747-7171},
     journal={J. Symbolic Comput.},
      volume={{\bf 24}},
      number={3-4},
       pages={235\ndash 265},
         url={http://dx.doi.org/10.1006/jsco.1996.0125},
        note={Computational algebra and number theory (London, 1993)},
}

\bib{BostonEllenberg}{article}{
      author={Boston, N.},
      author={Ellenberg, J.S.},
       title={Pro-p groups and towers of rational homology 3-spheres},
        date={2006},
     journal={Geom. Top.},
      volume={10},
       pages={331\ndash 334},
}

\bib{BMRS1}{article}{
      author={Bridson, M.R.},
      author={McReynolds, D.B.},
      author={Reid, A.W.},
      author={Spitler, R.},
       title={Absolute profinite rigidity and hyperbolic geometry},
        date={2020},
     journal={Annals of Math.},
      volume={{\bf 192}},
       pages={679\ndash 719},
}

\bib{BR}{article}{
      author={Bridson, M.R.},
      author={Reid, A.W.},
       title={Profinite rigidity, fibering, and the figure-eight knot},
        date={2020},
     journal={What's Next?: The Mathematical Legacy of William P. Thurston, Annals of Math. Study 205, PUP},
       pages={45\ndash 64},
}

\bib{bridsonfixedpt}{article}{
      author={Bridson, M.R.},
       title={{Profinite isomorphisms and fixed-point properties}},
        date={2024},
     journal={Algebr. Geom. Topol.},
      volume={24},
       pages={4103\ndash 4114},
}

\bib{burton2025computingclosedessentialsurfaces}{misc}{
      author={Burton, Benjamin~A.},
      author={Tillmann, Stephan},
       title={Computing closed essential surfaces in 3-manifolds},
        date={2025},
         url={https://arxiv.org/abs/1812.11686},
}

\bib{BurtonTillmann2025ComputingHaken}{article}{
      author={Burton, Benjamin~A.},
      author={Tillmann, Stephan},
       title={Computing closed essential surfaces in 3-manifolds},
        date={2025},
        ISSN={2367-1726,2367-1734},
     journal={J. Appl. Comput. Topol.},
      volume={9},
      number={3},
       pages={Paper No. 18, 42},
         url={https://doi.org/10.1007/s41468-025-00208-w},
}

\bib{BoyerZhang}{article}{
      author={Boyer, S.},
      author={Zhang, X.},
       title={{On Culler-Shalen seminorms and Dehn filling}},
        date={1998},
     journal={Annals of Math.},
      volume={{\bf 148}},
       pages={737–801},
}

\bib{CalegariDunfield}{article}{
      author={Calegari, F.},
      author={Dunfield, N.M.},
       title={Automorphic forms and rational homology 3-spheres},
        date={2006},
     journal={Geom. Top.},
      volume={10},
       pages={295\ndash 329},
}

\bib{SnapPy}{misc}{
      author={Culler, M.},
      author={Dunfield, N.M.},
      author={Goerner, M.},
      author={Weeks, J.R.},
       title={Snap{P}y, a computer program for studying the geometry and topology of $3$-manifolds},
         how={Available at http://snappy.computop.org (10/11/2021)},
}

\bib{CHWeeks}{article}{
      author={Callahan, P.J.},
      author={Hildebrand, M.V.},
      author={Weeks, J.R.},
       title={{A census of cusped hyperbolic 3-manifolds}},
        date={1999},
     journal={Math. Comp.},
      volume={68},
       pages={321\ndash 332},
}

\bib{Cthesis}{thesis}{
      author={Cheetham-West, T.},
       title={{Finite quotients of hyperbolic 3-manifold groups}},
        type={Ph.D. Thesis},
        date={2023},
        note={unpublished thesis},
}

\bib{CLRS}{article}{
      author={Cheetham-West, T.},
      author={Lubotzky, A.},
      author={Reid, A.W.},
      author={Spitler, R.},
       title={{Property FA is not a profinite property}},
        date={2025},
     journal={Groups, Geom. Dyn.},
      volume={{\bf 19(3)}},
       pages={1081\ndash 1087},
}

\bib{CY}{article}{
      author={Cheetham-West, T.},
      author={Yao, Y.},
       title={{Finite covers and strict boundary slopes of cusped hyperbolic 3-manifolds}},
        date={2025},
     journal={arXiv \tt{arXiv:2506.12289}},
}

\bib{DFPR}{article}{
      author={Dixon, J.D.},
      author={Formanek, E.W.},
      author={Poland, J.C.},
      author={Ribes, L.},
       title={Profinite completions and isomorphic finite quotients},
        date={1982},
     journal={J. Pure Appl. Algebra},
      volume={23(3)},
       pages={227\ndash 231},
}

\bib{DunfieldThurstonHakenExperiment}{article}{
      author={Dunfield, Nathan~M.},
      author={Thurston, William~P.},
       title={The virtual {H}aken conjecture: experiments and examples},
        date={2003},
        ISSN={1465-3060,1364-0380},
     journal={Geom. Topol.},
      volume={7},
       pages={399\ndash 441},
         url={https://doi.org/10.2140/gt.2003.7.399},
}

\bib{DunfieldHakenHWCensus}{misc}{
      author={Dunfield, Nathan~M.},
        date={1999},
         url={https://nmd.web.illinois.edu/slides/haken_slides.pdf},
        note={University of Warwick, Slide talk},
}

\bib{GilleSzamuelyCSA}{book}{
      author={Gille, P.},
      author={Szamuely, T.},
       title={{Central Simple Algebras and Galois Cohomology}},
      series={Cambridge Studies in Advanced Mathematics},
   publisher={Cambridge University Press},
        date={2006},
}

\bib{ptilmanngarden}{article}{
      author={Garden, G.},
      author={Tilmann, S.},
       title={{An invitation to Culler-Shalen theory in arbitrary characteristic}},
        date={2024},
     journal={arXiv \tt{2411.06859v1}},
}

\bib{HassHomology}{article}{
      author={Hass, J.},
       title={{Surfaces minimizing area in their homology class and group actions on 3-manifolds}},
        date={1988},
     journal={Math. Z.},
      volume={{\bf 199}},
       pages={501\ndash 509},
}

\bib{hatcherAT}{book}{
      author={Hatcher, Allen},
       title={{Algebraic Topology}},
   publisher={Cambridge University Press, Cambridge},
        date={2002},
        ISBN={0-521-79160-X; 0-521-79540-0},
}

\bib{Hempel3mfds}{book}{
      author={Hempel, J.},
       title={3-manifolds},
      series={AMS Chelsea Publishing Series},
   publisher={AMS Chelsea Pub.},
        date={2004},
}

\bib{Hempel}{article}{
      author={Hempel, J.},
       title={Some 3-manifold groups with the same finite quotients},
        date={2014},
     journal={{arXiv \tt{1409.3509}}},
         url={https://arxiv.org/abs/1409.3509},
}

\bib{HempelJacoExtensions}{article}{
      author={Hempel, John},
      author={Jaco, William},
       title={Fundamental groups of {$3$}-manifolds which are extensions},
        date={1972},
        ISSN={0003-486X},
     journal={Ann. of Math. (2)},
      volume={95},
       pages={86\ndash 98},
         url={https://doi.org/10.2307/1970856},
}

\bib{HodgsonWeeksCensus}{article}{
      author={Hodgson, Craig~D.},
      author={Weeks, Jeffrey~R.},
       title={Symmetries, isometries and length spectra of closed hyperbolic three-manifolds},
        date={1994},
        ISSN={1058-6458,1944-950X},
     journal={Experiment. Math.},
      volume={3},
      number={4},
       pages={261\ndash 274},
         url={http://projecteuclid.org/euclid.em/1048515809},
}

\bib{JSJ}{book}{
      author={Johannson, K.},
       title={{Homotopy equivalences of 3-manifolds with boundaries}},
   publisher={Lect. Notes Math., 761. Springer, Berlin},
        date={1979},
}

\bib{JS}{article}{
      author={Jaco, W.},
      author={Shalen, P.B.},
       title={{A new decomposition theorem for irreducible sufficiently-large 3-manifolds. Algebraic and geometric topology}},
        date={1978},
     journal={Proc. Sympos. Pure Math., Stanford Univ., Stanford, Calif., 1976), Part 2, pp. 71–84, Proc. Sympos. Pure Math., XXXII, Amer. Math. Soc., Providence, R.I.},
       pages={71\ndash 84},
}

\bib{JZ}{article}{
      author={Jaikin-Zapirain, A.},
       title={{Recognition of being fibered for compact 3-manifolds}},
        date={2020},
     journal={Geom. Topol.},
      volume={{\bf 24(1)}},
       pages={409\ndash 420},
}

\bib{KN}{article}{
      author={Kitano, T.},
      author={Nozaki, Y.},
       title={{Finiteness of the image of the Reidemeister torsion of a splice}},
        date={2020},
     journal={Ann. Math. Blaise Pascal},
      volume={27(1)},
       pages={19\ndash 36},
}

\bib{LiuArithmeticDetection}{article}{
      author={Liu, Y.},
       title={{Finite Quotients, Arithmetic Invariants, and Hyperbolic Volume}},
        date={2021},
     journal={Peking Math. J.},
      volume={3},
}

\bib{LiuAlmostProfiniteRigidity}{article}{
      author={Liu, Y.},
       title={{Finite-volume hyperbolic 3-manifolds are almost determined by their finite quotient groups}},
        date={2023},
     journal={Invent. Math.},
      volume={{\bf 231(2)}},
       pages={741\ndash 804},
}

\bib{Martelli}{book}{
      author={Martelli, B.},
       title={An introduction to geometric topology},
   publisher={CreateSpace Independent Publishing Platform},
        date={2016},
}

\bib{milneCFT}{misc}{
      author={Milne, J.S.},
       title={Class field theory (v4.03)},
        date={2020},
        note={Available at www.jmilne.org/math/},
}

\bib{MilnorS2}{article}{
      author={Milnor, J.},
       title={{A Unique Decomposition Theorem for 3-Manifolds}},
        date={1962},
     journal={Am. J. Math.},
      volume={{\bf 84(1)}},
       pages={1\ndash 7},
}

\bib{MaclachlanReidBook}{book}{
      author={Maclachlan, C.},
      author={Reid, A.W},
       title={{The Arithmetic of Hyperbolic 3-Manifolds}},
   publisher={Graduate Texts in Mathematics {\bf 219}, Springer-Verlag, New York},
        date={2000},
}

\bib{CullerShalen}{article}{
      author={Marc, C.},
      author={Shalen, P.B.},
       title={Varieties of group representations and splittings of 3-manifolds},
        date={1983},
     journal={Annals of Mathematics},
      volume={117},
      number={1},
       pages={109\ndash 146},
}

\bib{ns}{article}{
      author={Nikolov, N.},
      author={Segal, D.},
       title={{On finitely generated profinite groups, I: Strong completeness and uniform bounds}},
        date={2007},
     journal={Annals of Math.},
      volume={{\bf 165}},
       pages={171\ndash 238},
}

\bib{paoluzziporti}{article}{
      author={Paoluzzi, L.},
      author={Porti, J.},
       title={{Invariant character varieties of hyperbolic knots with symmetries}},
        date={2018},
     journal={Math. Proc. Camb. Philos. Soc.},
      volume={165(2)},
       pages={193 \ndash  208},
}

\bib{Purcell}{book}{
      author={Purcell, J.S.},
       title={{Hyperbolic knot theory}},
   publisher={AMS Graduate Studies in Mathematics {\bf 209}},
        date={2020},
}

\bib{reid2018profinite}{inproceedings}{
      author={Reid, Alan~W.},
       title={Profinite rigidity},
        date={2018},
   booktitle={Proceedings of the {I}nternational {C}ongress of {M}athematicians---{R}io de {J}aneiro 2018. {V}ol. {II}. {I}nvited lectures},
   publisher={World Sci. Publ., Hackensack, NJ},
       pages={1193\ndash 1216},
}

\bib{NickSnapPyNT}{misc}{
      author={Rouse, N.},
       title={snappy-nt},
   publisher={GitHub},
         how={\url{https://github.com/nicholasrouse/snappy-nt}},
        date={2022},
        note={Accessed: 2025-10-24},
}

\bib{ruberman}{article}{
      author={Ruberman, D.},
       title={{Seifert surfaces of knots in $S^4$}},
        date={1990},
     journal={Pacific J. Math.},
      volume={145(1)},
       pages={97\ndash 116},
}

\bib{RubermanOverflow}{misc}{
      author={Ruberman, D.},
       title={Hyperbolic homology spheres with infinite $\mathrm{SL}_2(\mathbb{C})$ character variety},
         how={MathOverflow},
         url={https://mathoverflow.net/q/442823},
        note={URL:https://mathoverflow.net/q/442823 (version: 2023-03-16)},
}

\bib{RibesZalesskiiBook}{book}{
      author={Ribes, L.},
      author={Zalesskii, P.A.},
       title={{Profinite Groups}},
   publisher={Ergeb. der Math. {\bf 40}, Springer-Verlag, Berlin},
        date={2000},
}

\bib{SerreTrees}{book}{
      author={Serre, J-P},
       title={{Trees}},
   publisher={Springer-Verlag, Berlin},
        date={1980},
}

\bib{StallingsSurface}{book}{
      author={Stallings, J.},
       title={{Group theory and three-dimensional manifolds. A James K. Whittemore Lecture in Mathematics given at Yale University, 1969.}},
   publisher={Yale Mathematical Monographs {\bf 4}, Yale University Press, New Haven},
        date={1980},
}

\bib{ScottWall}{inproceedings}{
      author={Scott, P.},
      author={Wall, T.},
       title={Topological methods in group theory},
        date={1979},
   booktitle={{Homological group theory (Proc. Sympos., Durham, 1977), London Math. Soc. Lecture Note Ser., {\bf 36}, Cambridge Univ. Press, Cambridge}},
       pages={137\ndash 203},
}

\bib{TaharaGL3Z}{article}{
      author={Tahara, K.},
       title={{On the finite subgroups of GL(3,$\Z$)}},
        date={1971},
     journal={Nagoya Math. J.},
      volume={41},
       pages={169\ndash 209},
}

\bib{sagemath}{manual}{
      author={{The Sage Developers}},
       title={{S}agemath, the {S}age {M}athematics {S}oftware {S}ystem ({V}ersion 10.4)},
        date={2024},
        note={{\tt https://www.sagemath.org}},
}

\bib{wilkesSFS}{article}{
      author={Wilkes, G.},
       title={{Profinite rigidity for Seifert fibre spaces}},
        date={2017},
     journal={Geom. Dedicata},
      volume={{\bf 188(1)}},
       pages={141–163},
}

\bib{WZ1}{article}{
      author={Wilton, H.},
      author={Zalesskii, P.},
       title={{Distinguishing geometries using finite quotients}},
        date={2017},
     journal={Geom. Topol.},
      volume={{\bf 21(1)}},
       pages={345\ndash 384},
}

\bib{WZ2}{article}{
      author={Wilton, H.},
      author={Zalesskii, P.},
       title={{Profinite detection of 3-manifold decompositions}},
        date={2019},
     journal={Compos. Math.},
      volume={{\bf 155(2)}},
       pages={246\ndash 259},
}

\end{biblist}
\end{bibdiv}

\end{document}